\documentclass[a4paper,11pt]{article}

\usepackage[english]{babel}
\usepackage[utf8x]{inputenc}
\usepackage[T1]{fontenc}
\usepackage[left, mathlines]{lineno}

\usepackage{placeins}
\usepackage{float}
\restylefloat{table}
\restylefloat{figure}

\usepackage[a4paper,top=2cm,bottom=2cm,left=2cm,right=2cm,marginparwidth=1.75cm]{geometry}

\usepackage{lmodern}
\usepackage[skip = 0pt]{caption}
\usepackage[aboveskip = -2pt]{subcaption}
\usepackage{mathtools, amsthm, amssymb}
\usepackage[dvipsnames]{xcolor}
\usepackage{stackengine, graphicx}
\usepackage{array}
\usepackage{calc}
\usepackage{dsfont}
\usepackage[nottoc, notlof, notlot]{tocbibind}
\usepackage{stmaryrd}

\usepackage{amsmath}
\usepackage{amsthm}
\usepackage{subcaption}
\usepackage{graphicx}
\usepackage[colorinlistoftodos]{todonotes}
\usepackage[colorlinks=true, allcolors=blue]{hyperref}
\usepackage{float}

\usepackage{authblk} 

\allowdisplaybreaks[4] 
\binoppenalty=\maxdimen
\relpenalty=\maxdimen

\theoremstyle{plain}
\newtheorem{thm}{Theorem}
\newtheorem{lem}[thm]{Lemma}
\newtheorem{prop}[thm]{Proposition}

\newtheorem{definition}[thm]{Definition}

\newtheorem{rem}[thm]{Remark}

\theoremstyle{remark}

\numberwithin{thm}{section}
\numberwithin{equation}{section}

\newenvironment{AMS}{\textbf{\textit{MSC 2020 Subject Classification:}}}{}
\newenvironment{keywords}{\textbf{\textit{Keywords:}}}{}
\newenvironment{acknowledgements}{\textbf{Acknowledgements}}{}
\newenvironment{headline}{\textbf{\textit{Running headline:}}}{}

\newcommand{\esp}{\mathbb{E}}
\newcommand{\proba}{\mathbb{P}}
\newcommand{\nmath}{\mathbb{N}}
\newcommand{\zmath}{\mathbb{Z}}

\newcommand{\rmath}{\mathbb{R}}

\newcommand{\bP}{\mathbf{P}}
\newcommand{\bE}{\mathbf{E}}

\newcommand{\rde}{\mathbb{R}^{2}}

\newcommand{\fcal}{\mathcal{F}}

\newcommand{\lcal}{\mathcal{L}}

\newcommand{\bcal}{\mathcal{B}}
\newcommand{\ecal}{\mathcal{E}}
\newcommand{\rcal}{\mathcal{R}}
\newcommand{\mcal}{\mathcal{M}}

\newcommand{\vcal}{\mathcal{V}}
\newcommand{\gcal}{\mathcal{G}}
\newcommand{\ccal}{\mathcal{C}}

\newcommand{\bfrak}{\mathfrak{B}}

\newcommand{\Vol}{\mathrm{Vol}}

\newcommand{\un}[1]{\mathds{1}_{\{#1\}}}

\newcommand{\dhor}{\sqrt{a^{2}\cos^{2}(\gamma)+b^{2}\sin^{2}(\gamma)}}

\newcommand{\tabg}{\mathbf{E}[D_{1}]}
\newcommand{\xte}{X_{t}^{express}}
\newcommand{\nte}{N_{t}^{express}}

\newcommand{\rtwo}{\mathbb{R}^{2}}
\newcommand{\ml}{\mcal_{\lambda}}
\newcommand{\setzeroun}{\{0,1\}}
\newcommand{\statespace}{D_{\ml}[0,+\infty)}
\newcommand{\lk}{\lcal_{\mu}^{k}}
\newcommand{\linfty}{\lcal_{\mu}^{\infty}}
\newcommand{\hp}{H}
\newcommand{\limeps}{\lim\limits_{\epsilon \to 0}}
\newcommand{\Jm}{J_{\mu}}

\title{Measure-valued growth processes in continuous space and growth properties starting from an infinite interface}

\author[1]{Apolline Louvet}
\author[2]{Amandine V{\'e}ber\footnote{Corresponding author}}

\affil[1]{University of Bath, Department of Mathematical Sciences, Bath BA2 7AY, UK}
\affil[2]{Universit\'e Paris Cit\'e, CNRS, MAP5, 75006 Paris, France
\vspace{0.5cm}

\begin{flushleft}
\footnotesize{Email addresses : apolline.louvet@polytechnique.edu, amandine.veber@parisdescartes.fr}
\end{flushleft}}

\begin{document}
\maketitle

\begin{abstract}
The $k$-parent and infinite-parent spatial Lambda-Fleming Viot processes (or SLFV), introduced in~\cite{louvet2023}, form a family of stochastic models for spatially expanding populations. These processes are akin to a continuous-space version of the classical Eden growth model (but with local backtracking of the occupied area allowed when $k$ is finite), while being associated with a dual process encoding ancestry.

In this article, we focus on the growth properties of the area occupied by individuals of type~$1$ (type $0$ encoding units of empty space). To do so, we first define the quantities that we shall use to quantify the speed of growth of the occupied area. Using the associated dual process and a comparison with a first-passage percolation problem, we show that the growth of the occupied region in the infinite-parent SLFV is linear in time. Because of the possibility of local backtracking of the occupied area, the result we obtain for the $k$-parent SLFV is slightly weaker. It gives an upper bound on the probability that a given location is occupied at time $t$, which also shows that growth in the $k$-parent SLFV is linear in time. We use numerical simulations to approximate the growth speed for the infinite-parent SLFV, and we observe that the actual speed may be higher than the speed expected from simple first-moment calculations due to the characteristic front dynamics.
\end{abstract}

\begin{headline}
Measure-valued growth processes
\end{headline}



\tableofcontents

\section{Introduction}\label{sec:intro}

Our aim in this work is to connect two worlds: measure-valued stochastic models of population dynamics, and set-valued stochastic growth models. More precisely, we shall be interested in the relation between a family of models introduced in \cite{louvet2023}, called the $k$-parent and $\infty$-parent spatial $\Lambda$-Fleming-Viot processes (SLFV), and the processes describing the dynamics of the region occupied by a given type of individuals in such models, seen as a stochastic growth model in continuous space.

Our motivation for this is twofold. First, we want to show that techniques from (discrete-space) first-passage percolation theory can be used to understand the long term dynamics of the \textit{occupied region} in the $k$-parent and $\infty$-parent SLFV. Second, we want to point out that measure-valued formulations of stochastic growth models provide a natural way of defining these models with initially occupied areas of infinite volume, such as a half-plane, and then studying the impact on the expansion of having an interface of infinite length (and therefore nontrivial fluctuations after some time) between the occupied and unoccupied regions.

\subsection{A family of measure-valued models of population expansion}\label{subsec:intro1}
Measure-valued Markov processes are now classical tools to model and analyse the dynamics of spatially structured populations \cite{bansaye2015,barton2010new,etheridge2000,fournier2004}, in particular when individuals have different genetic or phenotypic characters potentially impacting their reproduction and migration rates. In \cite{barton2010new}, a model called the spatial $\Lambda$-Fleming Viot process was introduced, in which individuals are uniformly spread over some finite dimensional space ($\mathbb{R}^d$, say) and take their (genetic) types from a compact set $K$. At each time $t\geq 0$, the local type distribution at each $z\in \mathbb{R}^d$ is given by a probability distribution $\rho_t(z,d\kappa)$ on $K$, and the state of the whole population is described by the measure
\begin{equation}\label{eqn:M}
M_t(dz,d\kappa) = \rho_t(z,d\kappa)dz
\end{equation}
on $\mathbb{R}^d\times K$. The state of the process $(M_t)_{t\geq 0}$ is locally updated at the times of a Poisson point process of ``reproduction'' events specifying which region in $\mathbb{R}^d$ is affected, and what is the fraction of the local population which is updated during the event (with the interpretation that a fraction of the individuals in the area is replaced by offspring of an individual chosen to reproduce). This model comes with a set of tools which are extensions of classical tools from population genetics and which make it amenable to analysis. In particular, it is possible to find a \textit{dual} process, interpreted as the genealogical process of a sample of individuals, whose law is linked to the law of $(M_t)_{t\geq 0}$ through a duality relation which is of great use for the analysis of the long term behaviour of $(M_t)_{t\geq 0}$ itself. We refer to Section~4 in \cite{barton2010new} for more details, and to Section~\ref{sec:defn_slfv} below for the precise form of the dual processes and of the duality relations that will be of interest in the present work. While \cite{barton2010new} focuses on the \textit{neutral} case where the single type chosen to ``reproduce'' is sampled uniformly at random in the area of the region affected by the reproduction event, in \cite{etheridge2020rescaling} a bias in favour of the reproduction of one type (out of the two possible types $0$ and $1$ in that work) is introduced by sampling \textit{two} types during the events called \textit{selective}, and deciding that the ``offspring'' are of type $0$ if and only if both sampled types (or ``potential parents'') are 0. The corresponding process is called the \textit{SLFV with fecundity selection}, and its dual process can again be considered as tracing back all the potential parents which may have influenced the type of a finite number of individuals sampled in the population at time $t$ (see Section~1.2 in \cite{etheridge2020rescaling}). Observe from this informal description that if we start an SLFV process from a configuration where, at any location in space, we find only one type (that is, each $\rho_0(z,d\kappa)$ is a Dirac mass at some $\kappa(z)\in K$), and if we assume that the fraction of local population replaced in all events is $1$, then at any time $t\geq 0$ we find only one type at any given location. Thereby, we obtain a model in which space is partitioned into regions occupied by different types, and the boundaries of these regions evolve in time due to the Poisson point process of events. This observation yields a natural connection with models of stochastic geometry and set-valued processes, and strongly suggests that methods from these fields can be used to obtain results on the dynamics of the regions encoded by the measures $M_t$.

Focusing again on the case $K=\{0,1\}$, assuming that the fraction of local population replaced during all events is $1$ and replacing the sampling of $2$ types by the sampling of $k\geq 2$ types during each reproduction event, we obtain the $k$-parent SLFV introduced in \cite{louvet2023} to model the dynamics of an expanding population (type~1 individuals being ``existing individuals'' and type~0 individuals being considered as ``non-existing'', or ``ghosts''). From now on, we restrict our attention to the most biologically relevant dimension $d=2$. The alternative constructions of the process $(M^{(k)})_{t\geq 0}$ are recalled in Sections~\ref{subsec:defn_martingale_problem} and~\ref{subsec:poisson_construction}; here we satisfy ourselves with the following informal construction. Since $K=\{0,1\}$ and since we anticipate that only one type of individuals is present at every location in space at any time, let us write (in the notation of~\eqref{eqn:M})
\begin{equation}\label{eqn:w}
w_t^{(k)}(z) := \rho_t^{(k)}(z,\{1\})\in \{0,1\}, \qquad z\in \mathbb{R}^2,\ t\geq 0.
\end{equation}
For each given $t$, the function $w_t^{(k)}:z\mapsto w_t^{(k)}(z)$ can be thought of as the indicator function of the area in $\mathbb{R}^2$ occupied by individuals of type~1. However, as in all variants of the spatial $\Lambda$-Fleming-Viot process, the kernel $z\mapsto \rho_t^{(k)}(z,d\kappa)$ (and consequently $w_t^{(k)}$) is defined up to Lebesgue null sets (that is, changing the value of $\rho_t^{(k)}$ at a Lebesgue null set of points does not change the value of the measure $M_t^{(k)}$, which is the real object of interest). With this observation, we shall abuse the language and refer to $w_t^{(k)}$ as a \textit{density} for $M_t^{(k)}$. Motivated by experimental work on expanding populations of microorganisms (see in particular Figure~4 in \cite{hallatschek2007genetic}, in which the geometric properties of the reproduction events impact the diversity observed in the population), we suppose that each event has an elliptical shape encoded by a triplet $(a,b,\gamma)\in (0,+\infty)^2\times (-\pi/2,\pi/2)$ giving its width, height and angle with respect to the $x$-axis (see Figure~\ref{fig:illustr_ellipse}). We fix a finite measure $\mu$ on $(0,+\infty)^2\times (-\pi/2,\pi/2)$ with compact support $S$, and we let $\Pi$ be a Poisson point process on $\mathbb{R}_+\times \mathbb{R}^2\times S$ with intensity measure $dt\otimes dz\otimes \mu(da,db,d\gamma)$. We start from $M_0$ of the form~\eqref{eqn:M}, and for every $(t,z,a,b,\gamma)\in \Pi$:
\begin{itemize}
\item We sample $k$ types $\kappa_1,\ldots,\kappa_k\in\{0,1\}$ independently according to the type distribution $$
    \frac{1}{\mathrm{Vol}(\bfrak_{a,b,\gamma}(z))}\int_{\bfrak_{a,b,\gamma}(z)} M_{t-}^{(k)}(dy,d\kappa) $$ at time $t-$ in the ellipse $\bfrak_{a,b,\gamma}(z)$ with parameters $(a,b,\gamma)$ and centred at $z$.
\item If at least one $\kappa_i$ is equal to $1$, then for every $z'\in \bfrak_{a,b,\gamma}(z)$ we set $w_t^{(k)}(z'):=1$. That is, we fill the area $\bfrak_{a,b,\gamma}(z)$ with individuals that are all of type~1.
\item If all $\kappa_i$, $1\leq i\leq k$, are equal to $0$, then for every $z'\in \bfrak_{a,b,\gamma}(z)$ we set $w_t^{(k)}(z'):=0$. That is, we fill the area $\bfrak_{a,b,\gamma}(z)$ with individuals that are all of type~0.
\item Outside $\bfrak_{a,b,\gamma}(z)$, nothing changes, and $M_t^{(k)}$ is set to be the measure with density $w_t^{(k)}$.
\end{itemize}
In the case where all event areas are balls (\textit{i.e.}, $a=b$), it was proved in \cite{louvet2023} that the process $(M_t^{(k)})_{t\geq 0}$ is well-defined, Markovian and that it can be characterised by means of a well-posed martingale problem. These properties extend to the case where reproduction events are elliptical that we consider here. In the interest of completeness (and to be able to refer to them in future work), we state these results in Section~\ref{sec:defn_slfv}.
\begin{rem} The assumption that $\mu$ should have finite support is not necessary to define the $k$-parent SLFV (see \cite{louvet2023}, where a less strict condition is formulated whose analogue for elliptical events applies here too). However, when we study the asymptotic growth of the region occupied by type~1 individuals, which is the main question in this paper, this assumption ensures that ``expansion'' events cover a bounded area, included in a box of large but fixed size. This fact then allows us to compare the growth process with an appropriate discrete space first-passage percolation model and to derive a lower bound on the expected time needed for the area occupied by type~$1$ individuals to reach a given point in space. See Section~\ref{sec:upper_bound} for more details.
\end{rem}

Thinking in terms of the area occupied by individuals of type~1 (which we shall call the \textit{occupied area} from now on), we see that this random set has a tendency to expand, but for a fixed value of $k$ it can also backtrack during an event in which the region affected contains both types of individuals (and hence overlaps the boundary of the occupied region) but, by chance, only type~0 individuals are sampled. However, the larger $k$, the more likely it is that at least one of the $k$ types $\kappa_i$ is equal to $1$ whenever the frequency of type~$1$ in the area is strictly positive. In the limit as $k\rightarrow \infty$, we therefore obtain a pure growth process, called the $\infty$-parent SLFV and denoted by $(M_t^{(\infty)})_{t\geq 0}$, whose dynamics are as follows. Let $\Pi$ be the same Poisson point process as earlier. We start from a measure $M_0$ of the form~\eqref{eqn:M} and for every $(t,z,a,b,\gamma)\in \Pi$:
\begin{itemize}
\item If $\mathrm{Vol}(\bfrak_{a,b,\gamma}(z)\cap \{y\in \mathbb{R}^2:\, w_{t-}^{(\infty)}(y)=1\})> 0$, then for every $z'\in \bfrak_{a,b,\gamma}(z)$ we set $w_t^{(\infty)}(z'):=1$. That is, we fill the area $\bfrak_{a,b,\gamma}(z)$ with individuals that are all of type~1.
\item Otherwise, for every $z'\in \bfrak_{a,b,\gamma}(z)$ we set $w_t^{(\infty)}(z'):=0$ (which modifies the value of the density at most in a set of locations of volume $0$).
\item Outside $\bfrak_{a,b,\gamma}(z)$, nothing changes, and $M_t^{(\infty)}$ is set to be the measure with density $w_t^{(\infty)}$.
\end{itemize}
Again, the process $(M_t^{\infty})_{t\geq 0}$ with ball-shaped events was proved to be well-defined, Markovian and solution to a well-posed martingale problem in \cite{louvet2023}. For completeness, in Section~\ref{sec:defn_slfv} we state the analogues of these results in the case of elliptical events.

\begin{rem}
Until now, the $k$-parent and $\infty$-parent SLFV have been defined with only two types of individuals, existing and non-existing. In this case, it is tempting to conclude that the set-valued dynamics of the occupied region is a sufficient description of the population process. However, to address questions raised by experimental and theoretical work on expanding colonies of microorganisms and \textit{gene surfing} \cite{gracia2013surfing,hallatschek2007genetic,hallatschek2010life}, in future work we shall define these models with a general, compact, type space $K$ and individuals of different types competing to invade space. In this case, the measure-valued formulation is a more natural and flexible framework than the formulation in terms of a random partitioning of space describing the regions occupied by the different subpopulations.
\end{rem}

\begin{rem} Although our focus here is on models of expanding populations, let us remark that the $k$-parent and $\infty$-parent SLFV can also be interpreted as modelling spatially homogeneous populations with strong selection in favour of type 1 individuals (the larger $k$, the stronger the selective advantage of type~1 over type~0 individuals). In this perspective, our study follows on from a long series of work focusing mainly on weak selection regimes (see, \textit{e.g.},
\cite{etheridge2017branchingbrownian,etheridge2017browniannet,etheridge2020rescaling}
for selection against one type,
\cite{forien2017central}
for general forms of selection,
\cite{biswas2018spatial,chetwynd2019rare,klimek2020spatial}
for fluctuating selection and
\cite{etheridge2017branching}
for selection against heterozygosity). Note however that in our regime of strong selection, no scaling of time or space is necessary.
\end{rem}

\subsection{Growth models starting from an occupied area of infinite volume}
Formulated as above, the $\infty$-parent SLFV corresponds to a continuous space version of the Eden model \cite{eden1961two} and is very similar to models of continuous-space first-passage percolation, and in particular to the model introduced by Deijfen in \cite{deijfen2003asymptotic}. Apart from their state spaces (Deijfen's process being set-valued), the main difference between the $\infty$-parent SLFV, and Deijfen's process and its multitype extensions studied in \cite{deijfen2003asymptotic,deijfen2004coexistence,deijfen2004stochastic}, is that in the latter, a reproduction event results in the expansion of the area occupied by type~1 individuals only when the centre of the event is already occupied (whereas in the $\infty$-parent SLFV, expansion occurs whenever there is a positive fraction of type~1 individuals in the area of the event). It is not clear how this subtle difference impacts the qualitative behaviour of the random growth, as a significant number of additional events occur in the $\infty$-parent SLFV. Note however that the expansion of the occupied region in the $\infty$-parent SLFV is obviously faster than in Deijfen's model.

A second, and perhaps more important, difference is that the $\infty$-parent SLFV is defined for very general initial conditions, which are not restricted to occupied regions of finite volume. In particular, this is why the definition provided in the previous section is only informal: when
$$
\mathrm{Vol}(\{z\in \mathbb{R}^2:\, w_t^{(\infty)}(z)=1\}) = + \infty,
$$
an infinite number of reproduction events intersect the occupied area in any time interval, and the global jump rate of the process is infinite. This difficulty is handled by endowing the set of measures of the form~\eqref{eqn:M} with the topology of vague convergence. In this way, only the changes in compact areas of $\mathbb{R}^2$ are used to define the dynamics of $(M_t^{(\infty)})_{t\geq 0}$, and since these changes occur at finite rate, it gives rise to a well-defined process characterised as the solution to a well-posed martingale problem. This construction is detailed in Sections~\ref{subsec:state_space} and~\ref{subsec:defn_martingale_problem}. In \cite{gouere2007,gouere2008}, an alternative construction of Deijfen's model is proposed, in which the time and centre coordinates of the points in the Poisson point process $\Pi$ are used to allocate ``contamination times'' to all other points than the centre in the area of the events (representing the time needed for an infection to travel from the centre to any other point in the area of the event). A point $y\in \mathbb{R}^d$ then lies in the occupied area at time $t$ if there is a finite path allowing to travel from the initial occupied area to $y$ before time $t$. This construction, very reminiscent of the formulation of discrete-space first-passage percolation models in terms of edge-passing times, allows arbitrary initially occupied areas. It has many merits, in particular it enables us to use the fine techniques from percolation theory to analyse the growth dynamics. However, it is specific to Deijfen's model and does not seem easily generalisable to other event-based expansion rules (including our own), especially when we want to allow local backtracking of the occupied area as in the $k$-parent SLFV.

In \cite{deijfen2003asymptotic,deijfen2004stochastic,gouere2008}, a shape theorem is proved for Deijfen's process $(E_t)_{t\geq 0}$ starting from a ball centred at $0$. More precisely, it is shown that when the distribution of the radii of the events has compact support \cite{deijfen2003asymptotic}, or has an exponential moment \cite{deijfen2004stochastic}, the scaled occupied area $t^{-1}E_t$ converges almost surely as $t$ tends to infinity to an asymptotic shape, which by isotropy is a ball. In \cite{gouere2008}, the authors obtain a necessary and sufficient condition on the distribution of event radii for such a shape theorem to hold true. We shall instead focus on growth starting from a half-plane $H$ and on the expansion of the occupied region in the direction transverse to the boundary of $H$. We shall similarly be interested in the asymptotic behaviour of the $k$-parent SLFV, in which the occupied area can also backtrack (which is not the case in the $\infty$-parent SLFV).

To place our work in the more global context already hinted at, first-passage percolation processes are stochastic growth models originally defined on a lattice (generally $\mathbb{Z}^{d}, \, d \geq 2$). In these models, each vertex is either occupied, or empty. If vertex $x \in \mathbb{Z}^{d}$ becomes occupied at time~$t$, and if it is connected to a vertex $y \in \mathbb{Z}^{d}$ by an edge (denoted $e$), then vertex~$y$ becomes occupied at time~$t + \tau_{e}$, where $\tau_{e}$ is independent from one edge to another (but not necessarily identically distributed).
Whether the expansion is linear in time depends on the distribution of the time needed to pass through any given edge of the grid (see, \textit{e.g.}, \cite{auffinger201750}) and whether one considers short-range percolation, such as nearest neighbour percolation, or long-range percolation, in which two vertices $x,y \in \mathbb{Z}^{d}$ are connected by an edge no matter the distance between them \cite{chatterjee2016multiple,cox1981some,richardson1973random}. In particular, when edge passing times are distributed as in the Eden growth model, growth is linear in time for short-range percolation and potentially faster for long-range percolation, depending on the relation between the distribution of $\tau_{e}$ and the distance between the two vertices it connects \cite{auffinger201750,chatterjee2016multiple}. Part of the proof of our main results will make use of the linear expansion of the occupied area in an appropriate first-passage percolation model, thanks to a natural coupling argument (see Section~\ref{sec:upper_bound}). Works like \cite{deijfen2003asymptotic,deijfen2004stochastic,howard1997euclidean}
are interesting examples of extensions of first-passage percolation models to a continuous setting.

When the growth is linear in time, in general it is possible to obtain lower and upper bounds on the speed of growth (see, \textit{e.g.}, \cite{alm2002lower,van1993inequalities}), or to use simulations to approximate it \cite{alm2015first}. For other stochastic growth models, such as the corner growth model \cite{seppalainen2009lecture}, which belongs to the family of \textit{last-passage percolation models}, it is possible to obtain an explicit speed of growth for specific passage time distributions \cite{rost1981non}. As many other growth models, the Eden model is conjectured to belong to the universality class of the Kardar-Parisi-Zhang (KPZ) equation \cite{kardar1986dynamic}. This equation generates rough fronts, whose characteristics are similar to the ones of fronts observed in some expanding biological populations (see, \textit{e.g.}, \cite{huergo2010morphology}). Such a conjecture is notably difficult to establish; to our knowledge, it has only been demonstrated in the case of the solid-on-solid growth model in \cite{bertini1997stochastic}.

\subsection{Linear growth of the occupied region in the $k$-parent and $\infty$-parent SLFV}\label{subsec:main results}
As we already explained, in what follows we suppose that, initially, the occupied region is the half-plane
\begin{equation}\label{eqn:def H}
H= \{(x,y) \in \rtwo : x < 0\}.
\end{equation}
That is, the initial value of each $k$-parent SLFV ($k\in \{2,3,\ldots,\infty\}$) is taken to be the measure $M^H$ defined as follows:
\begin{equation}\label{eqn:def MH}
M^{H}(dz,d\kappa) := \big(\mathds{1}_{H}(z)\delta_{1}(d\kappa) + \mathds{1}_{H^c}(z)\delta_{0}(d\kappa)\big)dz,
\end{equation}
and we write $w^H=\mathds{1}_{H}$ for its density. We let $(M_t^{(k)})_{t\geq 0}$ be the $k$-parent SLFV starting from $M^H$, and let $(w_t^{(k)})_{t\geq0}$ be its density process.

We are interested in the expansion transverse to the interface $\{(0,y),\, y\in \mathbb{R}\}$ of the area initially occupied by type~1 individuals and, to measure it, we consider the first time at which the point $(x,0)$ is occupied for $x>0$. Because we work with measures, whose densities are defined up to Lebesgue null sets, we need a definition of occupancy of $(x,0)$ which does not depend on the choice of density for $M_t^{(k)}$. To this end, let $\bcal_\epsilon(z)$ be the closed ball of radius $\epsilon$ centred at $z$, and $V_\epsilon$ be the volume of this ball.
\begin{definition}\label{defn:occupied_area}
Let $M$ be a measure of the form~\eqref{eqn:M}, with density $w_M$. We say that $z \in \rtwo$ is \textit{occupied in $M$} if
\begin{equation*}
\lim\limits_{\epsilon \to 0} V_{\epsilon}^{-1} \int_{\bcal_{\epsilon}(z)} w_{M}(z')dz' > 0,
\end{equation*}
and that $z$ is \textit{empty in $M$} otherwise.
\end{definition}
The fact that the limit in Definition \ref{defn:occupied_area} exists is a consequence of the fact that $w_{M}(z')dz'$ is absolutely continuous with respect to Lebesgue measure, and so the Lebesgue differentiation theorem applies.

First, we focus on the growth of the occupied region in the $\infty$-parent SLFV. To do so, let us introduce the following random times.
\begin{definition}\label{defn:time_location_reached}
For all $x > 0$, let
\begin{equation*}
\overrightarrow{\tau}\!_{x}^{\,(\infty)} := \min\left\{
t \geq 0 : \limeps V_{\epsilon}^{-1} \int_{\bcal_{\epsilon}((x,0))} w_{t}^{(\infty)}(z)dz > 0
\right\}.
\end{equation*}
\end{definition}
The fact that $\overrightarrow{\tau}\!_{x}^{\,(\infty)}$ can be defined as a minimum rather than an infimum is a consequence of Lemma~\ref{lem:property_1_density_slfv} below (which states that in all $k$-parent SLFV, the value of $M^{(k)}$ in a given compact area in $\mathbb{R}^2$ is updated at a bounded rate), and of the fact that the trajectories of $M^{(\infty)}$ are c\`adl\`ag by construction. Since a chain of events leading to $(x,0)$ does not need to pass through all points $(x',0)$ with $x'<x$, the function $x \to \overrightarrow{\tau}\!_{x}^{\,(\infty)}$ is not necessarily increasing. See Figure~\ref{fig:taux_forwards} for an illustration. Our first main result is the following.

\begin{figure}[t]
\centering
\includegraphics[width = 0.7\linewidth]{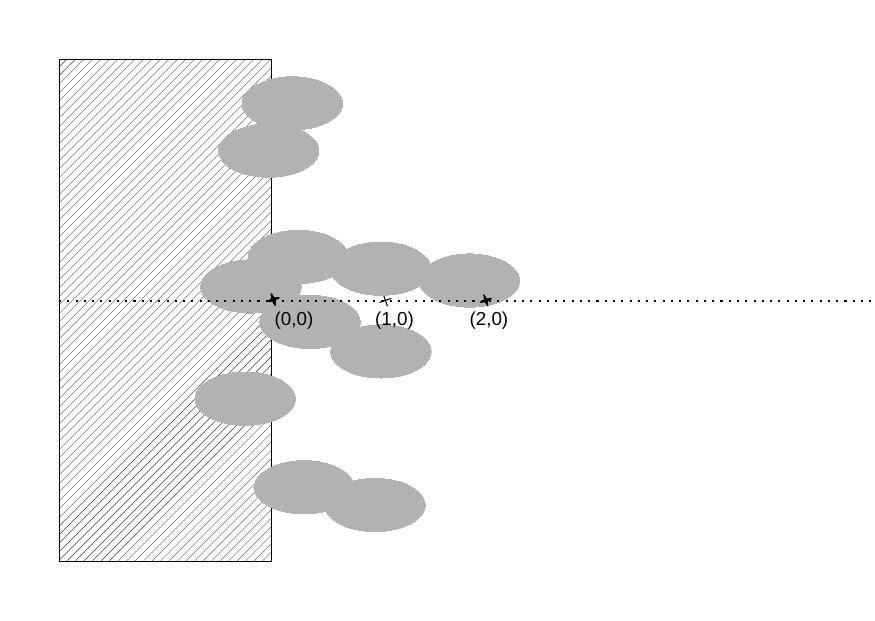}
\caption{State at some time $t$ of the $\infty$-parent SLFV. Initially, type~$1$ individuals cover the half-plane $H$, corresponding to the hatched area. Each grey ellipse represents a reproduction event occurring during the time interval $[0,t]$ which overlaps an area initially empty, and resulting in the corresponding area being completely filled with type~$1$ individuals. In this configuration, the point $(2,0)$ is already occupied at time $t$ while the point $(1,0)$ is not.}
\label{fig:taux_forwards}
\end{figure}

\begin{thm}\label{thm:speed_growth_infty_slfv}
There exists $\nu^{(\infty)} > 0$ such that
\begin{equation*}
\lim\limits_{x \to + \infty} \frac{\esp\big[\overrightarrow{\tau}\!_{x}^{\,(\infty)}\big]}{x} = \nu^{(\infty)}\quad \text{and} \quad
\lim\limits_{x \to + \infty} \frac{\overrightarrow{\tau}\!_{x}^{\,(\infty)}}{x} = \nu^{(\infty)} \quad\text{in probability.}
\end{equation*}
\end{thm}
The proof of Theorem~\ref{thm:speed_growth_infty_slfv} requires a series of comparison arguments and estimates that are presented in Sections~\ref{sec:reformulation}, \ref{sec:lower_bound} and \ref{sec:upper_bound}. It is completed at the end of Section~\ref{sec:upper_bound}. The proof also provides a lower bound on the speed of growth $(\nu^{(\infty)})^{-1}$ in terms of the speed of the ``express chain'' defined in Section~\ref{subsec:express chain}, see Remark~\ref{rem:speed in example}. However, the numerical exploration of the growth dynamics of the $\infty$-parent SLFV performed in Section~\ref{sec:numerical_simulations} shows that this lower bound can be highly suboptimal.

Let us now consider the dynamics of the occupied region in the $k$-parent SLFV. Due to the possible local backtracking of the boundary of the occupied region, the dynamics of the $k$-parent SLFV is significantly different from models of continuous first-passage percolation. In particular, we are no longer able to use the fine subadditivity argument of~\cite{liggett1985improved} to show convergence of the normalised hitting time of $(x,0)$. Instead, we obtain the following result.
\begin{thm}\label{thm:speed_growth_k_slfv}
Let $k\geq 2$. There exists $\nu^{(k)} \geq \nu^{(\infty)}$ such that for all $\varepsilon> 0$, there exists $x(\varepsilon)>0$ satisfying that for every $x\geq x(\varepsilon)$ and every $t>0$,
\begin{equation*}
\proba\left(
\limeps V_{\epsilon}^{-1} \int_{\bcal_{\epsilon}((x,0))} w_{t}^{(k)}(z)dz\, =\, 0
\right) \leq \frac{\big(\nu^{(k)}+\varepsilon\big)x}{t}.
\end{equation*}
\end{thm}
Note that the result of Theorem~\ref{thm:speed_growth_k_slfv} is weaker than that of Theorem~\ref{thm:speed_growth_infty_slfv} but essentially says that for large $x$, the point $(x,0)$ belongs to the occupied area after a time at most of the order of $\nu^{(k)} x$.

The rest of the paper is laid out as follows. In Section~\ref{sec:defn_slfv}, we define the $k$-parent and $\infty$-parent SLFV with elliptical reproduction events rigorously, introduce their dual processes, state the duality property which relates the SLFV to its dual process, and finally we propose an alternative construction of the SLFV processes which allows us to couple all $k$-parent SLFV ($k\in \{2,3,\ldots,+\infty\}$). In Section~\ref{sec:reformulation}, we use the duality relation obtained in Section~\ref{subsubsection:duality_relation} to reformulate Theorems~\ref{thm:speed_growth_infty_slfv} and \ref{thm:speed_growth_k_slfv} in terms of the dual processes of the $k$-parent SLFV. The proof of Theorem~\ref{thm:speed_growth_infty_slfv}, based on this reformulation, spans two sections. In Section~\ref{sec:lower_bound}, we show that the growth is \textit{at least} linear in time, and provide a lower bound on the speed of growth of the process. In Section~\ref{sec:upper_bound}, we use a comparison between the dual process and a first-passage percolation process to show that the growth is \textit{at most} linear in time. Combining the results from both sections yields Theorem~\ref{thm:speed_growth_infty_slfv}. In Section~\ref{sec:k-parent}, we use an analogous strategy to prove Theorem~\ref{thm:speed_growth_k_slfv}. Finally, in Section~\ref{sec:numerical_simulations} we use numerical simulations to approximate the growth speed of the area occupied by type~$1$ individuals, and we compare it to the lower bound obtained in Section~\ref{sec:lower_bound}. This bound was initially conjectured to be equal to (or at least a good approximation for) the growth speed, but the simulations show that the actual speed can be significantly higher. Moreover, they suggest that the growth is driven by ``spikes'' which occur at the front and then thicken in all directions. Providing a precise description of the impact of these fluctuations of the interface on the growth speed remains an open question.

\section{Spatial $\Lambda$-Fleming Viot processes with elliptical events}\label{sec:defn_slfv}

In this section, we rigorously define the $k$-parent and $\infty$-parent SLFV with elliptical events and their dual processes, adapting the different constructions carried out in \cite{louvet2023} in the case where all events are ball-shaped.

\subsection{State space and notation for ellipses}\label{subsec:state_space}
Let us first formalise the set of measures of the form~\eqref{eqn:M} in which our SLFV processes will take their values. Let $\mcal$ be the set of all measures on $\mathbb{R}^2\times \{0,1\}$ whose marginal distribution over $\mathbb{R}^2$ is Lebesgue measure. Such measures can be decomposed as in~\eqref{eqn:M}:
\begin{equation}\label{decomposition}
M(dz,d\kappa)= \rho_M(z,d\kappa)dz = \big(w_M(z)\delta_1(d\kappa) + (1-w_M(z))\delta_0(d\kappa)\big)dz,
\end{equation}
where for almost all $z\in \mathbb{R}^2$, $\rho_M(z,d\kappa)$ is a probability distribution on $\{0,1\}$ and $w_M(z)=\rho_M(z,\{1\})$. In this decomposition, the mapping $z\mapsto \rho_M(z,d\kappa)$ is defined up to Lebesgue null sets: $\rho_M$ and $\rho'_M$ are equivalent if and only if they give rise to the same measure $M$ on $\mathbb{R}^2\times \{0,1\}$, and so are the corresponding \textit{densities} $w_M=\rho_M(\cdot,\{1\})$ and $w'_M=\rho'_M(\cdot,\{1\})$. In general, $w_M$ takes its values in the interval $[0,1]$. The state space of our SLFV processes will be the subset of $\mcal$ of all measures admitting a density $w_M$ taking values in $\{0,1\}$ (that is, for every $z\in \mathbb{R}^2$, $\rho_M(z,d\kappa)$ is either a Dirac mass at $0$ or a Dirac mass at $1$):
\begin{equation}\label{eqn:Mlambda}
\mcal_\lambda:= \big\{M\in \mcal\,:\, \exists w_M:\mathbb{R}^2\rightarrow \{0,1\} \hbox{ satisfying }\eqref{decomposition}\big\}.
\end{equation}
We endow~$\ml$ with the topology of vague convergence. The space of all c\`adl\`ag $\mcal_\lambda$-valued trajectories is denoted by $\statespace$ and we equip this space with the standard Skorokhod topology.

\begin{rem}
In~\cite{louvet2023}, $w_M$ was defined as the local proportion of \textit{type~0} individuals, so that the occupied region was the set of locations at which $w_M(z)=0$. To work with more natural notation, here we define $w_M$ as the local proportion of type~1 individuals (corresponding to $1-w_M$ in \cite{louvet2023}) and the occupied region is the set of locations at which $w_M(z)=1$.
\end{rem}

Let us also set the notation regarding ellipses.
\begin{definition}\label{defn:notation_ellipse}
Let $z_{c} = (x_{c},y_{c}) \in \rde$, $(a,b) \in (0,+\infty)^{2}$ and $\gamma \in (-\pi/2,\pi/2)$. The ellipse with centre $z_{c}$ and parameters $(a,b,\gamma)$, denoted by $\bfrak_{a,b,\gamma}(z_{c})$, is defined by:
\begin{equation*}
\bfrak_{a,b,\gamma}(z_{c}) = \left\{\begin{pmatrix}
x_{c} \\
y_{c}
\end{pmatrix}
+ A_{\gamma} \begin{pmatrix}
ar\cos(\theta) \\
br \sin(\theta)
\end{pmatrix} : r \in [0,1], \theta \in [0,2\pi)
\right\},
\end{equation*}
where
\begin{equation*}
A_{\gamma} = \begin{pmatrix}
\cos(\gamma) & - \sin(\gamma) \\
\sin(\gamma) & \cos(\gamma)
\end{pmatrix}.
\end{equation*}
\end{definition}
See Figure~\ref{fig:illustr_ellipse} for an illustration. We denote the volume of an ellipse with parameters $(a,b,\gamma)$ by
\begin{equation}\label{defn:volume ellipse}
V_{a,b} := \text{Vol}(\bfrak_{a,b,\gamma}(0)).
\end{equation}
Note that this volume is independent of the angle $\gamma$.

\begin{figure}[t]
\centering
\includegraphics[width = 0.5\linewidth]{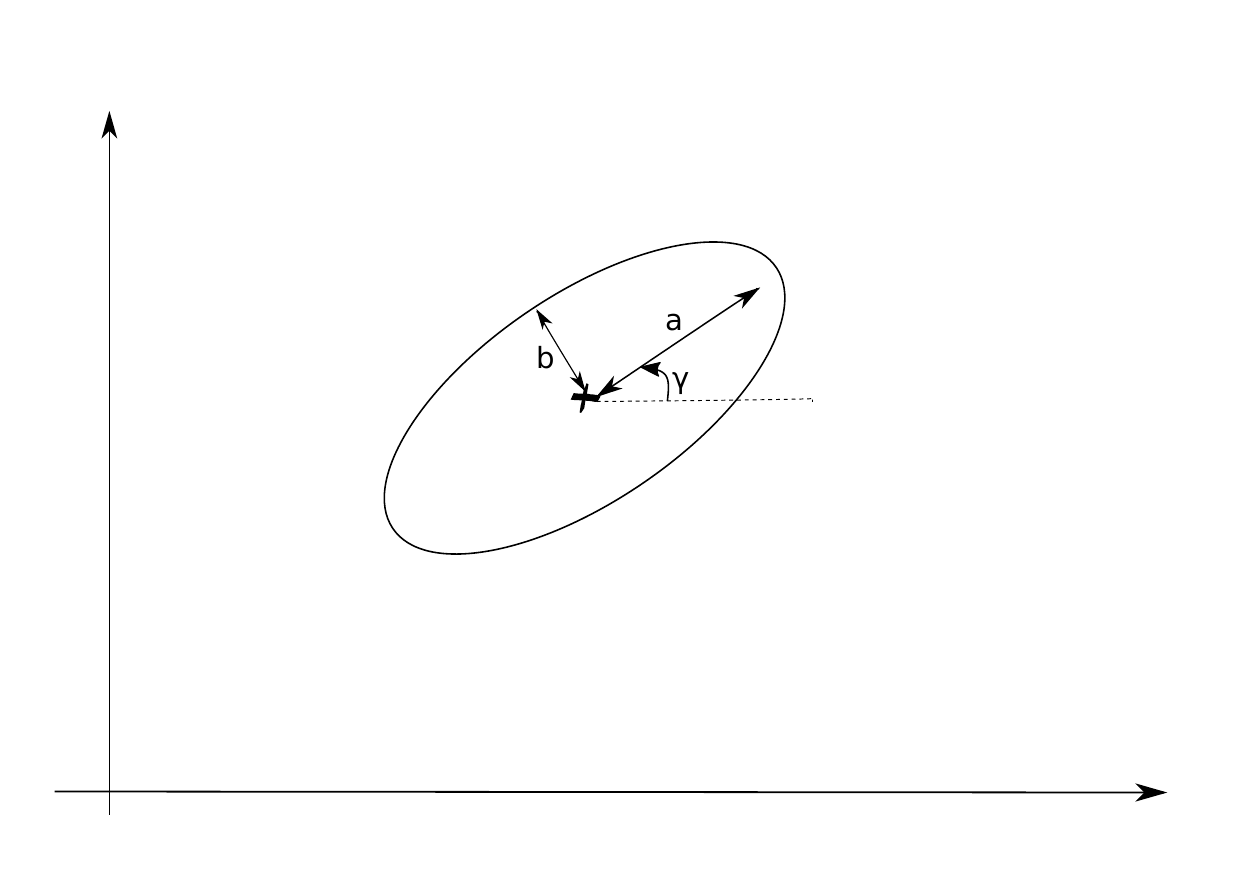}
\caption{Ellipse with parameters $(a,b,\gamma)$. }\label{fig:illustr_ellipse}
\end{figure}

\subsection{Definition as the unique solution to a martingale problem}\label{subsec:defn_martingale_problem}

\subsubsection{Martingale problem}\label{subsubsec:martingale_problem}
All the random objects we consider in this section are defined on some probability space $(\Omega, \fcal, \proba)$ and we let~$\esp$ denote the expectation with respect to~$\proba$.

First, let us introduce the test functions on which the martingale problems characterizing the $k$-parent and $\infty$-parent SLFV will be defined.
Let $C^{1}(\rmath)$ be the space of all continuously differentiable functions $F : \rmath \to \rmath$, and let $C_{c}(\rtwo)$ be the space of all continuous and compactly supported functions $f: \rtwo \to \mathbb{R}$. For all $f \in C_{c}(\rde)$ and $F \in C^{1}(\rmath)$, we define $\Psi_{F,f} : \mcal_{\lambda} \to  \rmath$ by
\begin{equation}\label{defn:test function}
\forall\, M \in \mcal_{\lambda},\qquad  \Psi_{F,f}(M) := F\left(\langle w_{M},f \rangle \right),
\end{equation}
where we recall that $w_{M}$ denotes a density of $M$ (fixed in an arbitrary way), and
\begin{equation*}
\langle w_M, f \rangle := \int_{\rde} f(z)w_M(z)dz.
\end{equation*}
Note that the value of $\Psi_{F,f}(M)$ does not depend on the choice of the density $w_M$. The set of functions that we shall use in the formulation of the martingale problems is then the set of functions of the form $\Psi_{F,f}$ with $F \in C^{1}(\mathbb{R})$ and $f \in C_{c}(\rtwo)$.

For all $(a,b) \in (0,+\infty)^{2}$,  $\gamma \in (-\pi/2,\pi/2)$ and $f \in C_{c}(\rtwo)$, let
\begin{equation}\label{defn:intersect}
\text{Supp}^{a,b,\gamma}(f) := \left\{
z \in \rde : \text{Vol}\left(\bfrak_{a,b,\gamma}(z) \cap \text{Supp}(f)\right) \neq 0
\right\},
\end{equation}
where $\text{Supp}(f)$ stands for the support of $f$. The set $\text{Supp}^{a,b,\gamma}(f)$ can be interpreted as the set of all potential centres $z \in \rde$ for ellipses with parameters $(a,b,\gamma)$ overlapping the support of $f$.

For all $z \in \rde$ and $w : \rde \to \{0,1\}$, let $\Theta_{z}^{a,b,\gamma}(w) : \rde \to \{0,1\}$ and $\overline{\Theta}_{z}^{a,b,\gamma}(w) : \rtwo \to \setzeroun$ be the functions defined by
\begin{align}
\Theta_{z}^{a,b,\gamma}(w) &:= \mathds{1}_{(\bfrak_{a,b,\gamma}(z))^{c}}w + \mathds{1}_{\bfrak_{a,b,\gamma}(z)}, \qquad \hbox{and}\nonumber \\
\text{and} \quad \overline{\Theta}_{z}^{a,b,\gamma}(w) &:= \mathds{1}_{(\bfrak_{a,b,\gamma}(z))^{c}}w. \label{defn:Theta}
\end{align}
That is, if $w$ is the indicator function of a set $E$ (representing the occupied area in our application), then $\Theta_{z}^{a,b,\gamma}(w)$ is the indicator function of the set $E\cup \bfrak_{a,b,\gamma}(z)$ and $\overline{\Theta}_{z}^{a,b,\gamma}(w)$ is the indicator function of the set $E \setminus \bfrak_{a,b,\gamma}(z)$ (which is equal to $E$ if $E\cap \bfrak_{a,b,\gamma}(z)=\emptyset$).

Let $\mu$ be a finite measure on $(0,+\infty)^{2} \times (-\pi/2,\pi/2)$ such that there exists $\rcal > 0$ satisfying
\begin{equation*}
\mu\Big(\big((0,\rcal] \times (0,\rcal]\big)^{c} \times (-\pi/2,\pi/2)\Big) = 0.
\end{equation*}
Let $\rcal_{\mu}$ be the smallest $\rcal > 0$ such that the above condition is satisfied, and set
\begin{equation}
S_{\mu} := (0,\rcal_{\mu}] \times (0,\rcal_{\mu}] \times (-\pi/2,\pi/2). \label{defn:support}
\end{equation}

We are now ready to introduce the operators $\lk$, $k \geq 2$, and $\linfty$ that we shall use to formulate the martingale problems. For every $F \in C^{1}(\rmath)$, $f \in C_{c}(\rde)$ and $M \in \ml$, let us define
\begin{align*}
\lk \Psi_{F,f}(M)
:= \int_{S_{\mu}}\int_{\text{Supp}^{a,b,\gamma}(f)}\int_{\bfrak_{a,b,\gamma}(z)^{k}} \frac{1}{V_{a,b}^{k}} \,  &\left[\left(\prod_{j = 1}^{k}\left(1-w_{M}(z_{j})\right)\right) \, F(\langle\overline{\Theta}_{z}^{a,b,\gamma}(w_{M}),f\rangle) \right.\\
&+ \left(1-\prod_{j = 1}^{k}(1-w_{M}(z_{j}))\right) \, F(\langle\Theta_{z}^{a,b,\gamma}(w_{M}),f\rangle)  \\
&- \left. F(\langle w_{M},f\rangle) \vphantom{\prod_{j = 1}^{k}}
\right]
dz_{1}\ldots dz_{k}dz\mu(da,db,d\gamma)
\end{align*}
and
\begin{align*}
\linfty \Psi_{F,f}(M):= \int_{S_{\mu}}\int_{\text{Supp}^{a,b,\gamma}(f)} & \left(
1 - \un{0}\left(
\int_{\bfrak_{a,b,\gamma}(z)}w_{M}(z')dz'
\right)\right) \\
&\times \left(
F\left(\langle \Theta_{z}^{a,b,\gamma}(w_{M}),f \rangle\right)
- F\left(\langle w_{M},f \rangle
\right)\right)dz\mu(da,db,d\gamma).
\end{align*}
These operators exactly encode the dynamics described informally in the introduction. Indeed, $\lk$ (\textit{resp.}, $\linfty$) is the generator of a jump process with rates given by the intensity measure of the Poisson point process used in the informal description of the $k$-parent SLFV (\textit{resp.}, the $\infty$-parent SLFV). Moreover, the product $\prod_j (1-w_M(z_j))$ in the formulation of $\lk$ is the indicator function that none of the $k$ `parental' locations sampled uniformly at random over $\bfrak_{a,b,\gamma}(z)$ during the event is already occupied, in which case the whole area of the event is subtracted from the occupied region (an operation encoded by $\overline{\Theta}_{z}^{a,b,\gamma}(w_M)$); otherwise, the area of the event is added to the occupied area, as described by $\Theta_{z}^{a,b,\gamma}(w_M)$. Likewise, in the formulation of $\linfty$, if the volume of occupied points in $\bfrak_{a,b,\gamma}(z)$ is zero then no changes occur; otherwise, the area of the reproduction event is added to the occupied area.

The following result enables us to give a rigorous definition of the $k$-parent and $\infty$-parent SLFV.

\begin{thm}\label{thm:definition_SLFV}
Let $M^{0}\in\ml$.
\begin{itemize}
\item[(i)] For all $k \geq 2$, there exists a unique $D_{\mcal_{\lambda}} [0,+\infty)$-valued process $(M_{t}^{(k)})_{t \geq 0}$ such that $M_{0}^{(k)} = M^{0}$ and, for all $F \in C^{1}(\rmath)$ and $f \in C_{c}(\rde)$,
\begin{equation}\label{MPk}
\left(\Psi_{F,f}\big(M_{t}^{(k)}\big) - \Psi_{F,f}\big(M_{0}^{(k)}\big)
- \int_{0}^{t} \lcal_{\mu}^{k} \Psi_{F,f}\big(M_{s}^{(k)}\big)ds
\right)_{t \geq 0}
\end{equation}
is a martingale. Moreover, the process $(M^{(k)}_t)_{t\geq 0}$ is Markovian.
\item[(ii)] There exists a unique $D_{\mcal_{\lambda}} [0,+\infty)$-valued process $(M_{t}^{(\infty)})_{t \geq 0}$ such that $M_{0}^{(\infty)} = M^{0}$ and, for all $F \in C^{1}(\rmath)$ and $f \in C_{c}(\rde)$,
\begin{equation}\label{MPinfty}
\left(\Psi_{F,f}\big(M_{t}^{(\infty)}\big) - \Psi_{F,f}\big(M_{0}^{(\infty)}\big)
- \int_{0}^{t} \lcal_{\mu}^{\infty} \Psi_{F,f}\big(M_{s}^{(\infty)}\big)ds
\right)_{t \geq 0}
\end{equation}
is a martingale. Moreover, the process $(M^{(\infty)}_t)_{t\geq 0}$ is Markovian.
\end{itemize}
\end{thm}
We then define the $k$-parent SLFV (\textit{resp.}, $\infty$-parent SLFV) associated with~$\mu$ and with initial condition~$M^{0}$ as the unique solution to the martingale problem $(\lk, \delta_{M^{0}})$ (\textit{resp.}, $(\linfty, \delta_{M^{0}})$).

The proof of Theorem~\ref{thm:definition_SLFV}~$(i)$ is identical to the proof of Theorem~2.1 in~\cite{louvet2023}, and therefore we omit it. The key observation is that the shapes of the reproduction events play no critical role in the definition of the process as long as their radii are controlled. It is the case here because we require that $\mu$ should be a finite measure with compact support. Theorem~\ref{thm:definition_SLFV}~$(ii)$ then follows from exactly the same construction of the process of densities $(w_{M_t^{(\infty)}})_{t\geq 0}$ via the sequential coupling of densities of $M^{(k)}$, $k\geq 2$, as in \cite{louvet2023} (during each event, we use the same $k$-sample of locations as for $w_{M^{(k)}}$, plus a $(k+1)$-th independent sample, to construct $w_{M^{(k+1)}}$ -- see Theorem~2.9 in \cite{louvet2023} and Section~\ref{subsec:poisson_construction} below). Again, because the proof in the case of ball-shaped events is long and technical and its adaptation to the case with elliptical events does not require new arguments, we do not repeat it.

\subsubsection{Duality relations}\label{subsubsection:duality_relation}
As in the case of ball-shaped events studied in \cite{louvet2023}, a key property of the $k$-parent and $\infty$-parent SLFV with elliptical events is the existence of a dual process, which is not only at the core of the strategy to prove uniqueness in law of the $D_{\mcal_\lambda}[0,\infty)$-valued solutions to the martingale problems stated in Theorem~\ref{thm:definition_SLFV}, but also provides a way to investigate properties of the expansion dynamics. These duality relations are stated in Propositions~\ref{prop:duality_relation_k_parent_SLFV} and \ref{prop:duality_relation} below. As often in such systems, the dual processes consist in tracing back in time the origins of the types contained in some compact area at a given time $t$, using the fact that the time-reversal of a Poisson point process with intensity measure $dt \otimes dz \otimes \mu(da,db,d\gamma)$ is again a Poisson point process with the same intensity measure. The duality formulae essentially state that the probability that a given set $E$ of points in $\rde$ contains no type~$1$ (``existing'') individuals at time $t$ is equal to the probability that, at time~$0$, there are no type~$1$ individuals in the whole set of locations which, at some time $s\in [0,t]$, have influenced the type composition in $E$ at time $t$. To see where this interpretation comes from, take $E=\{z_1,\ldots,z_l\}$ in Proposition~\ref{prop:duality_relation_k_parent_SLFV}, and $E=E^0$ (a finite union of measurable connected subsets of $\rde$ of finite volume) in Proposition~\ref{prop:duality_relation}.

Let us now proceed with the rigorous definition of the dual processes and the statement of the duality relations. These processes will be defined on a new probability space $(\mathbf{\Omega},\mathbf{\fcal},\bP)$, and $\bE$ denotes the expectation with respect to~$\bP$. Recall that $\mu$ is the finite measure with compact support on $(0,\infty)^2\times (-\pi/2,\pi/2)$ introduced just before \eqref{defn:support}.

\paragraph{The $k$-parent ancestral process.} The dual process associated with the $k$-parent SLFV is defined on the set $\mcal_{p}(\rmath^{2})$ of all finite counting measures on $\rtwo$. We equip this state space with the topology of weak convergence. For all $\Xi = \sum_{i = 1}^{l} \delta_{\xi_{i}} \in \mcal_{p}(\rtwo)$, for all $z \in \rtwo$, $(a,b) \in (0,+\infty)^{2}$ and $\gamma \in (-\pi/2,\pi/2)$, let us define
\begin{equation*}
I_{z,a,b,\gamma}(\Xi) := \big\{
i \in \{ 1,\ldots,l\}: \xi_{i} \in \bfrak_{a,b,\gamma}(z)
\big\}.
\end{equation*}
In words, $I_{z,a,b,\gamma}(\Xi)$ is the set of all indices of the points in $\Xi$ which lie in $\bfrak_{a,b,\gamma}(z)$.

\begin{definition}\label{defn:k dual process}
($k$-parent ancestral process) Let $k\in \mathbb{N}\setminus \{0,1\}$ and $\Xi^{0} \in \mcal_{p}(\rtwo)$. The $k$-parent ancestral process $(\Xi_{t}^{k})_{t \geq 0}$ associated with $\mu$ and with initial condition $\Xi^{0}$ is the $\mcal_{p}(\rtwo)$-valued Markov jump process defined as follows. Let $\overleftarrow{\Pi}$ be a Poisson point process on $\rmath_{+} \times \rtwo \times S_{\mu}$ with intensity measure $dt \otimes dz \otimes \mu(da,db,d\gamma)$ (defined on $(\mathbf{\Omega},\mathbf{\fcal},\bP)$), and let $\Xi_{0}^{k} = \Xi^{0}$. Then, for every $(t,z,a,b,\gamma) \in \overleftarrow{\Pi}$, writing
\begin{equation*}
\Xi_{t-}^{k} = \sum_{i = 1}^{N_{t-}^{k}} \delta_{\xi_{t-}^{k,i}}\,,
\end{equation*}
(where $N_{t-}^k$ denotes the number of atoms in $\Xi_{t-}^{k}$ and $\xi_{t-}^{k,i}$ the location in $\mathbb{R}^2$ of the $i$-th atom):
\begin{itemize}
\item If $I_{z,a,b,\gamma}(\Xi_{t-}^{k}) \neq \emptyset$, we sample $k$ points $z_{1},...,z_{k}$ independently and uniformly at random in $\bfrak_{a,b,\gamma}(z)$, and we set
\begin{equation*}
\Xi_{t}^{k} := \sum_{i = 1}^{N_{t-}^{k}}\delta_{\xi_{t-}^{k,i}} - \sum_{i \in I_{z,a,b,\gamma}(\Xi_{t-}^{k})} \delta_{\xi_{t-}^{k,i}} + \sum_{j = 1}^{k}\delta_{z_{j}}.
\end{equation*}
That is, we remove all the atoms of $\Xi_{t-}^{k}$ sitting in $\bfrak_{a,b,\gamma}(z)$, and we add $k$ atoms at locations that are i.i.d and uniformly distributed over the ellipse $\bfrak_{a,b,\gamma}(z)$.
\item If $I_{z,a,b,\gamma}(\Xi_{t-}^{k}) = \emptyset$, then no jump occurs and $\Xi_t^k=\Xi_{t-}^k$.
\end{itemize}
\end{definition}
This definition is analogous to Definition~2.6 in \cite{louvet2023} in the case of ball-shaped events, and the fact that it gives rise to a well-defined process follows again from the property that the number of atoms in $\Xi^k$ can be stochastically bounded by a Yule process with $k$ offspring and finite reproduction rate.

The $k$-parent SLFV $M^{(k)}$ and the $k$-parent ancestral process $\Xi^k$ associated with $\mu$ then satisfy the following duality relation. For $z_1,\ldots,z_l\in \rde$, we write $\Xi[z_{1},...z_{l}] := \sum_{i=1}^{l} \delta_{z_{i}}$.
\begin{prop}\label{prop:duality_relation_k_parent_SLFV}
Let $k \geq 2$. Let $M^{0} \in \ml$. Any solution $(M^{(k)}_t)_{t\geq 0}$ to the martingale problem~\eqref{MPk} satisfies that for every $l \in \nmath\setminus \{0\}$ and every integrable function $\psi$ on $(\rtwo)^{l}$, we have for all $t\geq 0$
\begin{align*}
\esp_{M^{0}}\Bigg[&
\int_{(\rtwo)^{l}} \psi(z_{1},...,z_{l})\left\{
\prod_{j = 1}^{l}\left(1-w_{M_{t}^{(k)}}(z_{j})\right)
\right\} dz_{1}...dz_{l}
\Bigg] \\
&= \int_{(\rtwo)^{l}} \psi(z_{1},...,z_{l}) \bE_{\Xi[z_{1},...,z_{l}]}\Bigg[
\prod_{j = 1}^{N_{t}^{k}}\left(
1-w_{M^{0}}(\xi_{t}^{k,j})
\right)\Bigg]dz_{1}...dz_{l}\,,
\end{align*}
where $N_{t}^{k}$ and $\xi_{t}^{k,j}$, $1 \leq j \leq N_{t}^{k}$, are such that $\Xi_{t}^{k} = \sum_{j = 1}^{N_{t}^{k}} \delta_{\xi_{t}^{k,j}}$.
\end{prop}
\begin{rem} Note that the quantities appearing on both sides of the equality in Proposition~\ref{prop:duality_relation_k_parent_SLFV} can be expressed as integrals with respect to $M_t^{(k)}$ and $M^0$, and so they do not depend on the choice of the densities $w_{M_{t}^{(k)}}$ and $w_{M^0}$.\end{rem}

This result can be seen as the exact analogue of Proposition~2.8 in \cite{louvet2023} once we remember that in~\cite{louvet2023}, $w_{M^{(k)}}$ denotes the frequency of \textit{type 0} (or ``non-existing'') individuals, which is $1-w_{M^{(k)}}$ in our notation. Again, we do not repeat the proof here.

\paragraph{The $\infty$-parent ancestral process.} Unlike the processes $\Xi^k$ taking their values in $\mcal_{p}(\rtwo)$, the dual process of the $\infty$-parent SLFV is defined on the set $\ecal^{cf}$ of all finite unions of measurable connected subsets of $\rtwo$ with positive and finite Lebesgue measure, that we endow with the topology induced by the Hausdorff distance. Recall the definition of the Poisson point process $\overleftarrow{\Pi}$ introduced in Definition~\ref{defn:k dual process}.

\begin{definition}\label{defn:dual_process} ($\infty$-parent ancestral process) Let $E^{0} \in \ecal^{cf}$. The $\ecal^{cf}$-valued $\infty$-parent ancestral process $(E_{t}^{\infty})_{t \geq 0}$ associated with $\mu$ and with initial condition $E^{0}$ is defined as follows. Let $E_{0}^{\infty} = E^{0}$. Then for every $(t,z,a,b,\gamma) \in \overleftarrow{\Pi}$,
\begin{itemize}
\item if $E_{t-}^{\infty} \cap \bfrak_{a,b,\gamma}(z)$ has non-zero Lebesgue measure, we set
\begin{equation*}
E_{t}^{\infty} = E_{t-}^{\infty} \cup \bfrak_{a,b,\gamma}(z);
\end{equation*}
\item otherwise, no jump occurs and $E_t^{\infty}=E_{t-}^{\infty}$.
\end{itemize}
\end{definition}

\begin{lem}
The process $(E_{t}^{\infty})_{t \geq 0}$ of Definition \ref{defn:dual_process} is well-defined and Markovian.
\end{lem}

\begin{proof}
Recall that for all $(t,z,a,b,\gamma) \in \overleftarrow{\Pi}$, we have $(a,b) \in (0,\rcal_{\mu}]^{2}$ (by definition of $\rcal_\mu$). Therefore,
\begin{equation*}
\bfrak_{a,b,\gamma}(z) \subseteq \bcal_{\rcal_{\mu}}(z) \text{ a.s.,}
\end{equation*}
where $\bcal_{\rcal_{\mu}}(z)$ is the ball of radius $\rcal_{\mu}$ centred at $z$,
and so we can bound the jump rate of $(E_{t}^{\infty})_{t \geq 0}$ from above by that of a $\infty$-parent ancestral process with the same initial condition and associated with $\mu(S_{\mu})\delta_{(\rcal_{\mu},\rcal_{\mu},0)}(da,db,d\gamma)$, which is finite (see Section~$4$ in \cite{louvet2023}).
\end{proof}

The $\infty$-parent SLFV and the $\infty$-parent ancestral process associated with the measure $\mu$ then satisfy the following duality relation.

\begin{prop}\label{prop:duality_relation}
Let $M^{0} \in \mcal_{\lambda}$. Any solution $(M^{(\infty)}_t)_{t\geq 0}$ to the martingale problem~\eqref{MPinfty} satisfies that for every $E^{0} \in \ecal^{cf}$, we have for all $t \geq 0$
\begin{equation}\label{eqn:duality}
\proba_{M^{0}}\bigg(\int_{E^{0}}w_{M_{t}^{(\infty)}}(z)dz = 0\bigg) = \bP_{E^{0}}\bigg( \int_{E_{t}^{\infty}}w_{M^{0}}(z) dz =0\bigg).
\end{equation}
\end{prop}
As mentioned at the beginning of the section, this duality relation can be interpreted as follows. Whenever a reproduction event affects one area, all the individuals in the area can be considered as \textit{potential parents}. If $\overleftarrow{\Pi}$ encodes the reproduction events affecting the population when going \textit{backwards in time}, then $E_{t}^{\infty}$ encodes the locations of the potential ancestors at time~$0$ of the individuals living in $E^{0}$ at time $t$. The duality relation then states that a given area $E^{0}$ contains only type 0 individuals at time $t$ (\textit{i.e.}, is ``empty'' at time~$t$) if and only if all their potential ancestors at time $0$ (located in $E_{t}^{\infty}$) are type 0 individuals. Note however that the duality relation only relates the \textit{laws} of the Markov processes $M^{(\infty)}$ and $E^\infty$, as the two processes for now are not even constructed on the same probability space.

\begin{rem}The careful reader will note that in \cite{louvet2023}, the dual process of the $\infty$-parent SLFV is formulated as an $\mcal_\lambda$-valued process instead of a set-valued process: instead of considering $(E_t^\infty)_{t\geq 0}$, the dual process is taken to be $(m(E_t^\infty))_{t\geq 0}$, where $m(E)=\mathds{1}_E(x)dx$. This choice resulted from the fact that one was interested in the convergence of the $k$-parent processes to their $\infty$-parent analogues as $k$ tended to infinity, for which it was more convenient to have all dual processes defined as finite measures on $\rde$. The two choices of encoding of the dual processes are equivalent for our purposes, as all formulae we exploit are integrals over the support of $E^\infty_t$. In this paper, we rather use the set-valued encoding, which alleviates the notation when we study properties of $E_t^\infty$.
\end{rem}

\subsection{Poisson point process-based construction}\label{subsec:poisson_construction}
As discussed right after the statement of Theorem~\ref{thm:definition_SLFV}~$(ii)$ in Section~\ref{subsec:defn_martingale_problem}, a solution to the martingale problem~\eqref{MPinfty} (and therefore a construction of the $\infty$-parent SLFV itself) was first obtained as the limit as $k\rightarrow \infty$ of a sequence of coupled $k$-parent SLFV. The coupling relies on a formalisation of the Poisson point process-based construction of SLFV described in the introduction, which was initially introduced in \cite{veber2015spatial} (for the ``classical'' SLFV process).

We present this construction here as it will allow us to show \textit{a priori} intuitive properties of the $k$-parent and $\infty$-parent SLFV which cannot be shown as easily with the martingale problem characterisation, namely, that:
\begin{enumerate}
\item The value taken by the process on a compact subset of $\mathbb{R}^{2}$ is updated at a finite rate.
\item Each (local) update of the process is done in an area shaped as an ellipse.
\item The occupied area in the $\infty$-parent SLFV can only increase as time goes on.
\end{enumerate}

For all $(a,b,\gamma) \in S_{\mu}$, let $\tilde{u}_{(a,b,\gamma)}$ be the law of a sequence of i.i.d. random variables $(\mathcal{P}_{n})_{n \geq 1}$ uniformly distributed over $\bfrak_{a,b,\gamma}(0)$.

Let $\Pi$ be a Poisson point process on $\mathbb{R}_{+} \times \mathbb{R}^{2} \times S_{\mu}$ with intensity measure $dt \otimes dz \otimes \mu(da,db,d\gamma)$, defined on the probability space $(\Omega, \mathcal{F}, \proba)$. We extend $\Pi$ by assigning to each point $(t,z,a,b,\gamma) \in \Pi$ a $\bfrak_{a,b,\gamma}(0)$-valued sequence $(p_{n})_{n \geq 1}$ sampled according to $\tilde{u}_{(a,b,\gamma)}$. The resulting point process will be denoted as $\Pi^{+}$, and will be referred to as the \textit{extended Poisson point process}. This Poisson point process encodes both the stochasticity in the location of reproduction events and in the choice of (the locations of) potential parents. Therefore, we can use it to provide a \textit{quenched} construction of the $k$-parent ancestral process.

\begin{definition}[Quenched $k$-parent ancestral process]\label{defn:quenched k-parent}
Let $k \geq 2$, let $\Xi^{0} \in \mcal_{p}(\mathbb{R}^{2})$, and let $\tilde{t} \geq 0$. The $k$-parent ancestral process $(\Xi_{k,t}^{\Pi^{+},\tilde{t},\Xi^{0}})_{0\leq t\leq \tilde{t}}$ associated with $\Pi^{+}$, started at time $\tilde{t}$ and with initial condition $\Xi^{0}$ is the $\mcal_{p}(\mathbb{R}^{2})$-valued Markov jump process defined as follows.
\begin{itemize}
\item First, we set $\Xi_{k,0}^{\Pi^{+},\tilde{t},\Xi^{0}} = \Xi^{0}$.
\item Then, for all $(t,z,a,b,\gamma,(p_{n})_{n \geq 1}) \in \Pi^{+}$ such that $t \leq \tilde{t}$, recalling that for $\Xi = \sum_{i = 1}^{l} \delta_{\xi_{i}} \in \mcal_{p}(\mathbb{R}^{2})$, $I_{z,a,b,\gamma}(\Xi) = \{i \in \{ 1,\ldots,l\}: \xi_{i} \in \bfrak_{a,b,\gamma}(z)\}$, if
\begin{equation*}
I_{z,a,b,\gamma}\left(\Xi_{k,(\tilde{t}-t)-}^{\Pi^{+},\tilde{t},\Xi^{0}}\right) \neq \emptyset,
\end{equation*}
then for all $1 \leq l \leq k$, we set
\begin{align*}
z_{l} &:= z + p_{l}\\
\text{and \quad\quad   }
\Xi_{k,\tilde{t}-t}^{\Pi^{+},\tilde{t},\Xi^{0}}
&:= \Xi_{k,(\tilde{t}-t)-}^{\Pi^{+},\tilde{t},\Xi^{0}}
- \sum_{z' \in I_{z,a,b,\gamma}\big(\Xi_{k,(\tilde{t}-t)-}^{\Pi^{+},\tilde{t},\Xi^{0}}\big)} \delta_{z'}
+ \sum_{l = 1}^{k} \delta_{z_{l}}.
\end{align*}
\end{itemize}
\end{definition}
For all $\Xi = \sum_{i = 1}^{l} \delta_{\xi_{i}} \in \mcal_{p}(\mathbb{R}^{2})$, we denote the set of atoms of $\Xi$ by
\begin{equation*}
A(\Xi) := \{
\xi_{i} : i \in \{ 1,\ldots,l\}
\}.
\end{equation*}
We can then construct the $k$-parent SLFV associated with $\mu$ and with initial condition $M^{0} \in \mcal_{\lambda}$. Let $\texttt{w} : \mathbb{R}^{2} \to \{0,1\}$ be a density of $M^{0}$ and let us set $w_{k,0}^{\Pi^{+},\texttt{w}} = \texttt{w}$. Then, for all $t>0$ and for all $z \in \rde$, let us set
\begin{equation}\label{eqn:densite_quenched_k_SLFV}
w_{k,t}^{\Pi^{+},\texttt{w}}(z) := 1-\left(\prod_{z' \in A\big(\Xi_{k,t}^{\Pi^{+},t,\delta_{z}}\big)} \left(1-\texttt{w}(z')\right)\right).
\end{equation}
Finally, for all $t \geq 0$, let us define
\begin{equation}\label{eqn:quenched SLFV}
M_{k,t}^{\Pi^{+},\texttt{w}}
:= \Big(w_{k,t}^{\Pi^{+},\texttt{w}}(z) \delta_{1}(d\kappa)
+ \big(1-w_{k,t}^{\Pi^{+},\texttt{w}}(z)\big) \delta_{0}(d\kappa)\Big)dz.
\end{equation}
Although each $w_{k,t}^{\Pi^{+},\texttt{w}}$ is a random \textit{function}, we shall call it the \textit{density} at time $t$ of the $k$-parent SLFV associated with $\Pi^{+}$ and with initial condition~$\texttt{w}$.

We have the two following results, whose proofs are identical respectively to the proofs of Lemma~3.5 and Theorem~2.14 in \cite{louvet2023}.
\begin{lem}\label{lem:poisson_constr_k_slfv}
The $D_{\mcal_{\lambda}}[0,+\infty)$-valued process \textit{$(M_{k,t}^{\Pi^{+},\texttt{w}})_{t \geq 0}$} defined above is solution to the martingale problem~\eqref{MPk}. Therefore, it is distributed like the $k$-parent SLFV.
\end{lem}

Since, as $k$ increases, we use an increasing number of locations during each event to decide the value of each $w_{k,t}^{\Pi^{+},\texttt{w}}(z)$, this sequence is almost surely nondecreasing for all $t \geq 0$ and $z \in \mathbb{R}^{2}$ and we can set
\begin{equation}\label{defn:w infty}
w_{\infty,t}^{\Pi^{+},\texttt{w}}(z) = \lim \limits_{k \to + \infty} w_{k,t}^{\Pi^{+},\texttt{w}}(z),
\end{equation}
In fact, since the only source of randomness in the construction of these densities is the common sequence of events associated with $\Pi^+$, we even have that with probability~1, such a limit exists simultaneously for all $t$ and $z$. Then, for every $t \geq 0$ we define
\begin{equation}\label{defn:M infty}
M_{\infty,t}^{\Pi^{+},\texttt{w}} := \Big(w_{\infty,t}^{\Pi^{+},\texttt{w}}(z)\delta_{1}(d\kappa)
+ \big(1-w_{\infty,t}^{\Pi^{+},\texttt{w}}(z)\big)\delta_{0}(d\kappa)\Big)dz.
\end{equation}

\begin{lem}\label{lem:poisson_constr_infty_slfv}
The $D_{\mcal_{\lambda}}[0,+\infty)$-valued process $(M_{\infty,t}^{\Pi^{+},\texttt{w}}(z))_{t\geq 0}$ defined above is solution to the martingale problem~\eqref{MPinfty}.
\end{lem}

Additionally, the densities $(w_{k,t}^{\Pi^{+},\texttt{w}})_{t \geq 0}$, $k \geq 2$ and $(w_{\infty,t}^{\Pi^{+},\texttt{w}})_{t \geq 0}$ satisfy the following properties, which have straightforward implications for the corresponding measure-valued evolutions.

\begin{lem}\label{lem:property_1_density_slfv}
Let $E \subseteq \rtwo$ be a bounded subset of $\rtwo$. For any $k\in \{2,3,\ldots,+\infty\}$, the value of \textit{$(w_{k,t}^{\Pi^{+},\texttt{w}})_{t \geq 0}$} over $E$ is updated at a bounded rate.

Moreover, if $E$ has non-zero Lebesgue measure and $(M_{t}^{(k)})_{t \geq 0}$ is the $k$-parent SLFV associated with $\mu$ and with initial condition $M^{0}$, then the value of $(\int_{E} w_{M_{t}^{(k)}}(z)dz)_{t \geq 0}$ is updated at a finite rate.
\end{lem}

\begin{proof}
The first part of Lemma~\ref{lem:property_1_density_slfv} is a consequence of the fact that $E$ is affected by reproduction events in $\Pi^{+}$ at a finite rate which is independent of the time interval we consider. The second part is a consequence of Lemmas~\ref{lem:poisson_constr_k_slfv} and~\ref{lem:poisson_constr_infty_slfv} along with uniqueness of the solution to the martingale problems~\eqref{MPk} and \eqref{MPinfty} which ensures that $M_{k,\cdot}^{\Pi^{+},\texttt{w}} $ and $M^{(k)}$ have the same distribution.
\end{proof}

\begin{lem}\label{lem:property_2_density_slfv}
Let $k\in \{2,3,\ldots,+\infty\}$. Let $z \in \rtwo$ and suppose that there exists $\epsilon_{0} > 0$ such that for all $z' \in \bcal_{\epsilon_{0}}(z)$, \textit{$\texttt{w}(z') = \texttt{w}(z)$}. Then, for every $t \geq 0$, with probability~$1$ there exists $\epsilon(z,t) > 0$ such that for all $z' \in \bcal_{\epsilon(z,t)}(z)$, we have \textit{$w_{k,t}^{\Pi^{+},\texttt{w}}(z') = w_{k,t}^{\Pi^{+},\texttt{w}}(z)$}.

As a consequence, with probability~1 there exists $\tilde{\epsilon}(z,t)>0$ such that the quantity
\begin{equation*}
V_{\epsilon}^{-1}  \int_{\bcal_{\epsilon}(z)} w_{M_{t}^{(k)}}(z')dz'
\end{equation*}
seen as a function of $\epsilon$ is constant over $(0,\tilde{\epsilon}(z,t)]$, equal to $0$ or $1$. (Recall the $V_{\epsilon}$ is the volume of a ball of radius $\epsilon$.)
\end{lem}

\begin{proof}
By Lemma~\ref{lem:property_1_density_slfv}, for a given $\delta\in (0,\epsilon_0)$, the value of $(w_{k,s}^{\Pi^{+},\texttt{w}})_{s \geq 0}$ over the ball $\bcal_\delta(z)$ changes finitely many times over the time interval $[0,t]$ with probability $1$. Let $U$ be the union of the areas of these events which overlap $\bcal_\delta(z)$ but do not contain $z$ if such events exist, and $U=\emptyset$ otherwise. The probability that $z\in \partial U$ is $0$ (since event centres are sampled uniformly in space), and so with probability~1, $z$ belongs to the interior of $U^c$ and we can find $\epsilon(z,t)>0$ such that $\bcal_{\epsilon(z,t)}(z) \subset \bcal_{\delta}(z)\cap U^c$. By construction, each point $z'\in \bcal_{\epsilon(z,t)}(z)$ is affected by exactly the same events as $z$ during the time interval $[0,t]$ and since the values of the density at $z$ and $z'$ at time $0$ are equal, so are they at time $t$.

The second part of the lemma is a consequence of Lemmas~\ref{lem:poisson_constr_k_slfv} and~\ref{lem:poisson_constr_infty_slfv} giving equality in distribution of $M^{(k)}$ and $M_{k,t}^{\Pi^{+},\texttt{w}}$ (note that for each $t,z$, the variables $\epsilon(z,t)$ and $\tilde{\epsilon}(z,t)>0$ are only equal in distribution).
\end{proof}
\begin{rem}\label{rem:local constant}
Note from the proof of Lemma~\ref{lem:property_2_density_slfv} that if $z$ is overlapped by at least one reproduction event in $[0,t]$, then by construction the value of \textit{$w_{k,t}^{\Pi^{+},\texttt{w}}(z')$} is the same for all $z'\in \bcal_{\epsilon(z,t)}(z)$ (whatever the value of \textit{$\texttt{w}$}). The assumption that \textit{\texttt{w}} should be locally constant around $z$ in Lemma~\ref{lem:property_2_density_slfv} is thus only useful in the case (of positive probability) where no events overlap $z$ during the interval $[0,t]$, in which case \textit{$w_{k,t}^{\Pi^{+},\texttt{w}}$} and \textit{$\texttt{w}$} are equal over $\bcal_{\epsilon(z,t)}(z)$ (and so the former is a locally constant function only if the latter is).
\end{rem}

\subsection{The $\infty$-parent SLFV as a random growing set}\label{subsec:growing set}
The informal construction of the $\infty$-parent SLFV suggests that the occupied area is a random growing set. However, this is not a immediate consequence of any of the two possible constructions of the $\infty$-parent SLFV we just introduced. The goal of this section is to show how we can use these two complementary constructions to prove that the $\infty$-parent SLFV corresponds to a random growing set to some extent.

Recall the definition of occupancy of a point given in Definition~\ref{defn:occupied_area}. In the case of the $k$-parent and $\infty$-parent SLFV, by Lemma~\ref{lem:property_1_density_slfv}, the occupancy of a location in space is updated at a finite rate. Moreover, the limit appearing in Definition~\ref{defn:occupied_area} is generally $\{0,1\}$-valued, in the following sense.
\begin{lem}\label{lem:occupancy_0_1_valued}
Let $M^0 \in \ml$, and let $z \in \rtwo$. Assume that there exists $\epsilon_{0} > 0$ such that
\begin{equation}\label{eqn:lem_occupancy}
V_{\epsilon_{0}}^{-1}\int_{\bcal_{\epsilon_{0}}(z)} w_{M^{0}}(z')dz' \in \{0,1\}.
\end{equation}
For $k\in \{2,3,\ldots,+\infty\}$, let $(M_{t}^{(k)})_{t \geq 0}$ be the $k$-parent SLFV associated with $\mu$ and with initial condition $M^{0}$. Then, for all $t \geq 0$,
\begin{equation*}
\lim\limits_{\epsilon \to 0} V_{\epsilon}^{-1} \int_{\bcal_{\epsilon}(z)} w_{M_{t}^{(k)}}(z')dz' \in \{0,1\}
\quad
\text{a.s.}
\end{equation*}
\end{lem}

\begin{proof}
By (\ref{eqn:lem_occupancy}), we can choose $w_{M^{0}}$ in such a way that for all $z' \in \bcal_{\epsilon_{0}}(z)$, $w_{M^{0}}(z') = w_{M^{0}}(z)$. We can then conclude by using Lemma~\ref{lem:property_2_density_slfv}.
\end{proof}
Note that, here again, the assumption on $w_{M^0}$ is only used to ensure that the conclusion of the lemma holds true even if $z$ is not overlapped by an event during the interval $[0,t]$ (\textit{cf.} Remark~\ref{rem:local constant}).

Informally, the following lemma states that, once a location is reached by the expansion in the $\infty$-parent SLFV, it is still occupied at any later time with probability~1. Note that, as local backtracking of the area occupied by type~1 individuals is possible in the $k$-parent SLFV when $k<\infty$, we do not have an equivalent statement for $M^{(k)}$.

\begin{lem}\label{lem:random_growing_set}
Let $M^{0} \in \ml$. Let $(M_{t}^{(\infty)})_{t \geq 0}$ be the $\infty$-parent SLFV associated with $\mu$ and with initial condition $M^{0}$. For every $z \in \rtwo$ and $\epsilon > 0$, let
\begin{align*}
\tau_{z}^{(\epsilon)} &:= \min \left\{
s \geq 0 : V_{\epsilon}^{-1} \int_{\bcal_{\epsilon}(z)} w_{M_{s}^{(\infty)}}(z')dz' > 0
\right\} \\
\text{and } \quad \tau_{z} &:= \min \left\{
s \geq 0 : \lim\limits_{\epsilon ' \to 0} V_{\epsilon '}^{-1} \int_{\bcal_{\epsilon '}(z)} w_{M_{s}^{(\infty)}}(z')dz' > 0
\right\}.
\end{align*}
Then for all $z \in \rtwo$, $\epsilon > 0$ and $t \geq 0$, we have
\begin{align*}
\proba\left(\left.
V_{\epsilon}^{-1} \int_{\bcal_{\epsilon}(z)} w_{M_{t}^{(\infty)}}(z')dz' > 0\,
\right|
\tau_{z}^{(\epsilon)} \leq t
\right) &= 1 \\
\text{and} \quad \proba\left(\left.
\lim\limits_{\epsilon ' \to 0} V_{\epsilon '}^{-1} \int_{\bcal_{\epsilon '}(z)} w_{M_{t}^{(\infty)}}(z')dz' > 0\,
\right|
\tau_{z} \leq t
\right) &= 1.
\end{align*}
\end{lem}
Observe that $\tau_{z}^{(\epsilon)}$ and $\tau_z$ can be defined as minima and not infima thanks to the result of Lemma~\ref{lem:property_1_density_slfv} (by taking $E=\bcal_\epsilon(z)$ for instance) and the fact that the trajectories of $M^{(\infty)}$ are right-continuous by construction.

\begin{proof}
Let $\epsilon > 0$, $z \in \rtwo$ and $t \geq 0$. Let $\Pi^{+}$ be the extended Poisson point process defined in Section~\ref{subsec:poisson_construction},
let $\texttt{w} : \rtwo \to \{0,1\}$ be a density of $M^{0}$ and let $w_{\infty,t}^{\Pi^{+},\texttt{w}}$ be the density of $M_t^{(\infty)}$ given by (\ref{defn:w infty}). Recall that the densities $w_{\infty,t}^{\Pi^{+},\texttt{w}}$ are deterministic functions of $\Pi^+$ and $\texttt{w}$, and therefore we shall reason in terms of the probability that $\Pi^+$ generates an appropriate sequence of events (including the choice of the sequence of `parental' locations) for the results of Lemma~\ref{lem:random_growing_set} to hold true.

First, observe that since the value of the density in any given ball around $z$ is updated at finite rate by Lemma~\ref{lem:property_1_density_slfv}, the random times $\tau_{z}^{(\epsilon)}$ and $\tau_{z}$ necessarily belong to the set of times of events in $\Pi^+$ overlapping $\bcal_{\epsilon}(z)$, and the intersection of this set with a given interval $[0,t]$ is a.s. finite.

Let us work conditional on the event $\{\tau_{z}^{(\epsilon)} \leq t\}$. The probability that the intersection of $\bcal_{\epsilon}(z)$ and the area $\bfrak_{a,b,\gamma}$ of the event in $\Pi^+$ occurring at time $\tau_{z}^{(\epsilon)}$ is of positive volume is $1$, consequently we can a.s. find $\tilde{z}\in \bcal_{\epsilon}(z)$ and $\tilde{\epsilon}>0$ such that $\bcal_{\tilde{\epsilon}}(\tilde{z})\subset \bcal_{\epsilon}(z)\cap \bfrak_{a,b,\gamma}$ and $w_{\infty,\tau_{z}^{(\epsilon)}}^{\Pi^{+},\texttt{w}}$ is constant over $\bcal_{\tilde{\epsilon}}(\tilde{z})$.
Moreover, by definition of $\tau_{z}^{(\epsilon)}$, the event occurring at that time renders the fraction of occupied area in $\bcal_{\epsilon}(z)$ positive, and so necessarily it fills $\bcal_{\tilde{\epsilon}}(\tilde{z})$ with type~1 individuals. That is,
\begin{equation}\label{eqn:area_occupied}
w_{\infty,\tau_{z}^{(\epsilon)}}^{\Pi^{+},\texttt{w}}(z') = 1, \qquad \hbox{for all }z' \in \bcal_{\tilde{\epsilon}}(\tilde{z}).
\end{equation}
Furthermore, by construction of the sequence of densities $w_{k,\cdot}^{\Pi^{+},\texttt{w}}$ (considering one more `parental' location during each event to pass from $k$ to $k+1$) and their increasing limit $w_{\infty,\cdot}^{\Pi^{+},\texttt{w}}$, with probability~$1$ there exists $k_{0} \geq 2$ such that for all $z' \in \bcal_{\tilde{\epsilon}}(\tilde{z})$ and $k \geq k_{0}$,
\begin{equation}\label{eqn:approx}
w_{k,\tau_{z}^{(\epsilon)}}^{\Pi^{+},\texttt{w}}(z')
= w_{\infty,\tau_{z}^{(\epsilon)}}^{\Pi^{+},\texttt{w}}(z') = 1.
\end{equation}

Let us now consider the time interval $[\tau_{z}^{(\epsilon)},t]$. Again, the number of events in $\Pi^+$ whose area overlaps $\bcal_{\tilde{\epsilon}}(\tilde{z})$ occurring in this time interval is almost surely finite. During each of these events, parental locations are sampled in an i.i.d. manner, uniformly over the area of the event, and so a.s. there exists a finite $n$ (depending on the event considered) such that the $n$-th sampled location belongs to $\bcal_{\tilde{\epsilon}}(\tilde{z})$. Consequently, for each given realisation of $\Pi^+$ we can take the maximum of $k_0$ and of the finitely many integers $n$ given by each event overlapping $\bcal_{\tilde{\epsilon}}(\tilde{z})$. Calling $k^*$ this integer, we see that for every $k\geq k^*$, the ancestral process $(\Xi_{k,s}^{\Pi^{+},t,\delta_{z'}})_{s\in [0,t-\tau_{z}^{(\epsilon)}]}$ starting at time $t$ from any point $z'\in \bcal_{\tilde{\epsilon}}(\tilde{z})$ constantly contains at least one atom in $\bcal_{\tilde{\epsilon}}(\tilde{z})$ until $s=t-\tau_{z}^{(\epsilon)}$. Together with \eqref{eqn:approx} which ensures that we can find a sequence of atoms at ``occupied'' points in the ancestry of $z'$ when going back from (forward) time $\tau_{z}^{(\epsilon)}$ to $0$, this allows us to conclude that for such a realisation of $\Pi^+$, for every $k\geq k^*$ and $z'\in \bcal_{\tilde{\epsilon}}(\tilde{z})$, we have
$$
w_{k,t}^{\Pi^{+},\texttt{w}}(z') = w_{\infty,\tau_{z}^{(\epsilon)}}^{\Pi^{+},\texttt{w}}(z') = 1.
$$
As $\bcal_{\tilde{\epsilon}}(\tilde{z})\subset \bcal_{\epsilon}(z)$, we obtain that
$$
V_{\epsilon}^{-1} \int_{\bcal_{\epsilon}(z)} w_{M_{t}^{(\infty)}}(z')dz' > 0.
$$
Since this is true for any realisation of $\Pi^+$ such that $\tau_{z}^{(\epsilon)}\leq t$, the first part of Lemma~\ref{lem:random_growing_set} is shown.

The proof of the second property stated in Lemma~\ref{lem:random_growing_set} follows the same lines, with $\tilde{z}=z$ and the a.s. finite set of times of interest being now the times at which the point $z$ is overlapped by the area of an event in $\Pi^+$. We do not repeat it.
\end{proof}

\section{Growth dynamics at the front edge - A \textit{backwards in time} viewpoint}\label{sec:reformulation}

We now turn to our case study of growth from an infinite interface. Recall that $H$ is the half-plane
$$
H= \{(x,y) \in \rtwo : x < 0\}
$$
and that in all that follows, the initial condition for the different SLFV processes we consider is
\begin{equation*}
M^{H}(dz,d\kappa) := \big(\mathds{1}_H(z)\delta_{1}(d\kappa) + \mathds{1}_{H^c}(z)\delta_{0}(d\kappa)\big)dz.
\end{equation*}
We write $w^H=\mathds{1}_H$ for its density.

In this section, we reformulate the expected results of Theorems~\ref{thm:speed_growth_infty_slfv} and~\ref{thm:speed_growth_k_slfv} in terms of the dual processes $\Xi^k$ and $\Xi^\infty$, to prepare the ground for their proofs via standard tools from percolation theory in Sections~\ref{sec:lower_bound}, \ref{sec:upper_bound} and \ref{sec:k-parent}.

\subsection{A \textit{backwards in time} version of Theorem \ref{thm:speed_growth_infty_slfv}}\label{subsec:back infty}
First, we reformulate Theorem~\ref{thm:speed_growth_infty_slfv} in terms of the growth properties of the $\infty$-parent ancestral process. Recall that $\mathcal{E}^{cf}$ is the set of all finite unions of measurable connected subsets of $\rtwo$ with positive and finite Lebesgue measure. Because we shall need to start our ancestral process from single points $(x,0)$, which are \textit{not} an element of $\mathcal{E}^{cf}$ (since they do not have positive Lebesgue measure), let us introduce a slight generalisation of the $\infty$-parent ancestral process which allows us to take such points as initial conditions. This process is defined on the state space
\begin{equation*}
\widetilde{\mathcal{E}}^{cf} := \mathcal{E}^{cf} \cup \big\{\{z\} : z \in \mathbb{R}^{2}\big\}.
\end{equation*}

Recall also that $\overleftarrow{\Pi}$ is a Poisson point process with intensity measure $dt \otimes dz \otimes \mu(da, db, d\gamma)$ on $\mathbb{R}_{+} \times \mathbb{R}^{2} \times S_{\mu}$ (see Definition~\ref{defn:k dual process}).
\begin{definition}\label{defn:better_infty_dual}
Let $z_{0} \in \mathbb{R}^{2}$. The $\widetilde{\mathcal{E}}^{cf}$-valued $\infty$-parent ancestral process $(E_{t}^{\infty,z_0})_{t \geq 0}$ with initial condition $\{z_{0}\}$ and associated with $\mu$ is defined as follows.

First, we set $E_{0}^{\infty,z_{0}} = \{z_{0}\}$. Then, for all $(t,z,a,b,\gamma) \in \overleftarrow{\Pi}$:
\begin{itemize}
\item If $E_{t-}^{\infty,z_{0}} = \{z_{0}\}$ and $z_{0} \in \bfrak_{a,b,\gamma}(z)$, we set $E_{t}^{\infty,z_{0}} = \bfrak_{a,b,\gamma}(z)$.
\item If $E_{t-}^{\infty,z_{0}} \neq \{z_{0}\}$ and $E_{t-}^{\infty,z_{0}} \cap \bfrak_{a,b,\gamma}(z)$ has non-zero Lebesgue measure, we set
\begin{equation*}
E_{t}^{\infty,z_{0}} = E_{t-}^{\infty,z_{0}} \cup \bfrak_{a,b,\gamma}(z).
\end{equation*}
\end{itemize}
\end{definition}

As for the initial $\infty$-parent ancestral process, $(E_{t}^{\infty,z_{0}})_{t \geq 0}$ is well-defined and Markovian. Moreover, once it jumps for the first time, its behaviour is identical to that of the original $\infty$-parent ancestral process.

For all $x > 0$ and all $\epsilon > 0$, let $(E_{t}^{\epsilon,x})_{t \geq 0}$ be the $\infty$-parent ancestral processes with initial condition $\bcal_{\epsilon}((x,0))$ constructed by using the events of $\overleftarrow{\Pi}$. Moreover, let $(E_{t}^{x})_{t \geq 0}$ be the $\infty$-parent ancestral process with initial condition $\{(x,0)\}$ also constructed using $\overleftarrow{\Pi}$. That is, in the heavier notation of Definition~\ref{defn:better_infty_dual}, $E^x=E^{\infty,(x,0)}$. These coupled ancestral processes satisfy the following property.

\begin{lem}\label{lem:lem1}
For all $x > 0$, if $t_{0}^{x}$ is the first time at which $(x,0)$ is affected by a reproduction event (that is, if $t_{0}^{x}$ is the first time at which $(E_{t}^{x})_{t \geq 0}$ jumps), then a.s. there exists $\epsilon_{0}^{x} > 0$ such that for all $t \geq t_{0}^{x}$, we have
\begin{equation*}
\forall\, \epsilon < \epsilon_{0}^{x},\qquad E_{t}^{\epsilon,x} = E_{t}^{x}.
\end{equation*}
Furthermore, we have with probability~1
\begin{equation*}
\lim\limits_{\epsilon \to 0} \Vol\left(E_{t}^{\epsilon,x} \cap H\right) = \Vol\left(
E_{t}^{x} \cap H \right) \qquad \hbox{for all }t\geq 0.
\end{equation*}
\end{lem}
Concretely, $\epsilon_{0}^{x}$ is the radius of a ball around $(x,0)$ small enough to be included in the area of the first event overlapping $(x,0)$, and not to be overlapped by any event in $\overleftarrow{\Pi}$ occurring before $t_{0}^{x}$.
\begin{proof}
Let $x > 0$. By the same arguments as in the proof of Lemma~\ref{lem:property_2_density_slfv}, with probability~1 there exists $\epsilon_{0}^{x} > 0$ such that
\begin{align}
E_{t}^{\epsilon,x} &= E_{0}^{\epsilon,x} \qquad \hbox{for all } 0 < \epsilon \leq \epsilon_{0}^{x} \hbox{ and all } 0 \leq t < t_{0}^{x}, \hbox{ and} \label{eqn:lem_2_13} \\
E_{t_{0}^{x}}^{\epsilon,x} &= E_{t_{0}^{x}}^{x} \qquad \ \hbox{for all } 0 < \epsilon \leq \epsilon_{0}^{x}.  \nonumber
\end{align}
The first part of the lemma is then a consequence of the strong Markov property of the $\infty$-parent ancestral process applied at time $t_0^x$.

Then, on the event of probability~1 on which $\epsilon_{0}^{x}$ exists, considering radii small enough that $\bcal_\epsilon((x,0))\cap H=\emptyset= \{(x,0)\}\cap H$, we can separate two cases according to whether $t<t_{0}^{x}$ or $t\geq t_{0}^{x}$.
\begin{enumerate}
\item For $t \in [0,t_{0}^{x})$, for all $\epsilon \in (0,\min\{\epsilon_{0}^{x},x\})$ we have by (\ref{eqn:lem_2_13}):
$$
\Vol\left(E_{t}^{\epsilon,x} \cap H\right)= \Vol\left(\bcal_{\epsilon}((x,0)) \cap H\right)= 0 = \Vol\left(\{(x,0)\} \cap H\right)= \Vol\left(
E_{t}^{x} \cap H \right). $$
\item For $t \geq t_{0}^{x}$, the first part of Lemma~\ref{lem:lem1} tells us that $E_{t}^{\epsilon,x} = E_{t}^{x}$ for all $\epsilon \in (0,\epsilon_{0}^{x})$. Therefore,
\begin{equation*}
\Vol\left(E_{t}^{\epsilon,x} \cap H \right) = \Vol\left(E_{t}^{x} \cap H \right).
\end{equation*}
\end{enumerate}
This allows us to conclude.
\end{proof}

We can now introduce an equivalent of $\overrightarrow{\tau}\!_{x}^{\,(\infty)}$ for the $\infty$-parent ancestral process.
\begin{definition}\label{defn:tilde_tau_x}
For all $x > 0$, let
\begin{equation*}
\tilde{\tau}_{x}^{(\infty)} :=
\left\lbrace
\begin{array}{ll}
\min \left\{t \geq 0 :  \lim\limits_{\epsilon \to 0} \Vol\left(
E_{t}^{\epsilon,x} \cap H\right) > 0 \right\}  & \mbox{if there exists $\epsilon_{0}^{x}>0$ as in Lemma~\ref{lem:lem1},} \\
0 & \mbox{otherwise}.
\end{array}
\right.
\end{equation*}
\end{definition}

\begin{lem}\label{lem:tilde_tau_x} For all $x > 0$, $\tilde{\tau}_{x}^{(\infty)}$ is well-defined, and is almost surely equal to
\begin{align*}
&\min\left\{t \geq 0:
E_{t}^{x} \cap H \neq \emptyset
\right\}.
\end{align*}
In particular, $\tilde{\tau}_{x}^{(\infty)} \neq 0$ a.s.
\end{lem}

\begin{proof}
Let $x > 0$. By Lemma~\ref{lem:lem1}, $\tilde{\tau}_{x}^{(\infty)}$ is well-defined and almost surely equal to
\begin{equation*}
\min\Big\{ t \geq 0 : \Vol\left(E_{t}^{x} \cap H \right) > 0 \Big\} \geq t_{0}^{x} >0,
\end{equation*}
The fact that $t_{0}^{x}>0$ a.s. is a consequence of the fact that $\mu$ is a finite measure with compact support and so the point $(x,0)$ is overlapped by an event of $\overleftarrow{\Pi}$ at a finite rate. Now, since $H$ is an open set we have for all $(a,b,\gamma) \in S_{\mu}$
\begin{equation*}
\left\{z \in \mathbb{R}^{2} : \bfrak_{a,b,\gamma}(z) \cap H \neq \emptyset \text{ and } \Vol\left(
\bfrak_{a,b,\gamma}(z) \cap H\right) = 0\right\} = \emptyset,
\end{equation*}
and so no events can lead to a situation where $E_t^x\cap H\neq \emptyset$ and $\Vol(E_t^x\cap H)=0$. Therefore, we can write
\begin{equation*}
\min\big\{t \geq 0 : \Vol\left(E_{t}^{x} \cap H\right) > 0 \big\}
= \min\left\{t \geq 0 :
E_{t}^{x} \cap H \neq \emptyset
\right\} \text{ a.s.,}
\end{equation*}
which allows us to conclude.
\end{proof}

The random variable $\tilde{\tau}_{x}^{(\infty)}$ can be interpreted as the equivalent of $\overrightarrow{\tau}\!_{x}^{\,(\infty)}$ for the dual process in the following sense.
\begin{lem}\label{lem:dualite_tau_x}
For all $x > 0$, $\tilde{\tau}_{x}^{(\infty)}$ and $\overrightarrow{\tau}\!_{x}^{\,(\infty)}$ have the same distribution.
\end{lem}

\begin{proof}
Let $x > 0$. For all $\epsilon > 0$, let us set
\begin{align*}
\overrightarrow{\tau}\!_{x}^{\, \epsilon,(\infty)} &:= \min\bigg\{
t \geq 0 : V_{\epsilon}^{-1} \int_{\bcal_{\epsilon}((x,0))}w_{t}^{(\infty)}(z)dz > 0
\bigg\}, \qquad \hbox{and} \\
\tilde{\tau}\!_{x}^{\ \epsilon,(\infty)} &:= \min\big\{ t \geq 0 : \Vol\left(E_{t}^{\epsilon,x} \cap H \right) > 0 \big\}.
\end{align*}
If we can show that
\begin{align}
&\limeps \overrightarrow{\tau}\!_{x}^{\,\epsilon,(\infty)} = \overrightarrow{\tau}\!_{x}^{\,(\infty)} \quad \text{a.s.,} \label{eqn:cond_1_lemma_2_16} \\
&\limeps \tilde{\tau}\!_{x}^{\ \epsilon,(\infty)} = \tilde{\tau}\!_{x}^{\,(\infty)} \quad \text{a.s.,} \label{eqn:cond_2_lemma_2_16} \\
&\overrightarrow{\tau}\!_{x}^{\,\epsilon,(\infty)} \stackrel{(d)}{=} \tilde{\tau}\!_{x}^{\ \epsilon,(\infty)}\quad \hbox{for all } \epsilon > 0, \label{eqn:cond_3_lemma_2_16}
\end{align}
then we can conclude that $\tilde{\tau}\!_{x}^{\,(\infty)}$ and $\overrightarrow{\tau}\!_{x}^{\,(\infty)}$ have the same distribution.

Let us start with the proof of \eqref{eqn:cond_1_lemma_2_16}. For all $0 < \epsilon < \epsilon '$, as $w_{\overrightarrow{\tau}\!_{x}^{\,\epsilon,(\infty)}}^{(\infty)}$ is $\{0,1\}$-valued, we can write
$$
\int_{\bcal_{\epsilon '}((x,0))} w^{(\infty)}_{\overrightarrow{\tau}\!_{x}^{\,\epsilon,(\infty)}}(z)dz \geq \int_{\bcal_{\epsilon}((x,0))} w^{(\infty)}_{\overrightarrow{\tau}\!_{x}^{\,\epsilon,(\infty)}}(z)dz > 0
$$
by definition of $\overrightarrow{\tau}\!_{x}^{\,\epsilon,(\infty)}$. Therefore, $\overrightarrow{\tau}\!_{x}^{\,\epsilon',(\infty)} \leq \overrightarrow{\tau}\!_{x}^{\,\epsilon,(\infty)}$. In particular, this means that $\limeps \overrightarrow{\tau}\!_{x}^{\,\epsilon,(\infty)}$ exists.

Let $t \geq 0$. By Lemma~\ref{lem:property_2_density_slfv}, with probability $1$ there exists $\tilde{\epsilon}(z,t) > 0$ such that the function
\begin{equation*}
\epsilon \mapsto V_{\epsilon}^{-1} \int_{\bcal_{\epsilon}((x,0))} w_{t}^{(\infty)}(z)dz
\end{equation*}
is constant over $(0,\tilde{\epsilon}(z,t)]$, equal to a number $v\in\{0,1\}$. Consequently, we also have
$$
\lim_{\epsilon\rightarrow 0}V_{\epsilon}^{-1} \int_{\bcal_{\epsilon}((x,0))} w_{t}^{(\infty)}(z)dz = v.
$$
If $t\geq \overrightarrow{\tau}\!_{x}^{\,(\infty)}$, by the second part of Lemma~\ref{lem:random_growing_set} we know that
$$
1= \lim_{\epsilon\rightarrow 0}V_{\epsilon}^{-1} \int_{\bcal_{\epsilon}((x,0))} w_{t}^{(\infty)}(z)dz = V_{\epsilon}^{-1} \int_{\bcal_{\epsilon}((x,0))} w_{t}^{(\infty)}(z)dz \quad \hbox{for all }\epsilon \in (0,\tilde{\epsilon}(z,t)].
$$
Hence, $t\geq \overrightarrow{\tau}\!_{x}^{\,\epsilon,(\infty)}$ for all $\epsilon \in (0,\tilde{\epsilon}(z,t)]$, and using the monotonicity in $\epsilon$ of $\overrightarrow{\tau}\!_{x}^{\,\epsilon,(\infty)}$, we obtain that
$$
t\geq \lim_{\epsilon\rightarrow 0} \overrightarrow{\tau}\!_{x}^{\,\epsilon,(\infty)}.
$$
This implies that $\overrightarrow{\tau}\!_{x}^{\,(\infty)}\geq \lim_{\epsilon\rightarrow 0} \overrightarrow{\tau}\!_{x}^{\,\epsilon,(\infty)}$. Conversely, if $t<\overrightarrow{\tau}\!_{x}^{\,(\infty)}$ then
$$
0= \lim_{\epsilon\rightarrow 0}V_{\epsilon}^{-1} \int_{\bcal_{\epsilon}((x,0))} w_{t}^{(\infty)}(z)dz = V_{\epsilon}^{-1} \int_{\bcal_{\epsilon}((x,0))} w_{t}^{(\infty)}(z)dz \quad \hbox{for all }\epsilon \in (0,\tilde{\epsilon}(z,t)].
$$
By the first part of Lemma~\ref{lem:random_growing_set}, $ V_{\epsilon}^{-1} \int_{\bcal_{\epsilon}((x,0))} w_{t}^{(\infty)}(z)dz=0$ is equivalent to $t<\overrightarrow{\tau}\!_{x}^{\,\epsilon,(\infty)}$, and so necessarily, $t<\lim_{\epsilon\rightarrow 0} \overrightarrow{\tau}\!_{x}^{\,\epsilon,(\infty)}$. This gives us that $\overrightarrow{\tau}\!_{x}^{\,(\infty)}\leq \lim_{\epsilon\rightarrow 0} \overrightarrow{\tau}\!_{x}^{\,\epsilon,(\infty)}$. Combining the two inequalities, we obtain~\eqref{eqn:cond_1_lemma_2_16}.

Let us now prove \eqref{eqn:cond_2_lemma_2_16}. Constructing all the $\infty$-parent ancestral processes with the same Poisson point process ensures that $\epsilon \mapsto \tilde{\tau}\!_{x}^{\ \epsilon,(\infty)}$ is a.s. non-increasing and hence that $\limeps \tilde{\tau}\!_{x}^{\ \epsilon,(\infty)}$ exists with probability~1. By the proof of Lemma~\ref{lem:lem1}, with probability~1 there exists $\epsilon_{0}^x>0$ such that for all $t > 0$,
\begin{equation*}
\Vol\left(
E_{t}^{\epsilon,x} \cap H
\right) = \Vol\left(
E_{t}^{x} \cap H
\right) \qquad \hbox{for all } 0 < \epsilon \leq \epsilon_{0}^x.
\end{equation*}
This implies that
\begin{equation*}
\limeps \tilde{\tau}\!_{x}^{\ \epsilon,(\infty)} = \tilde{\tau}\!_{x}^{\,(\infty)} \quad \text{a.s.}
\end{equation*}

Finally, let us prove \eqref{eqn:cond_3_lemma_2_16}. Let $\epsilon > 0$ and $t \geq 0$. Using the first part of Lemma~\ref{lem:random_growing_set}, Proposition~\ref{prop:duality_relation}, the fact that $w_0^\infty=\mathds{1}_H$ and finally the fact that $(E_{t}^{\epsilon,x})_{t\geq 0}$ is non-decreasing for the inclusion, we can write
\begin{align*}
\proba\left(\overrightarrow{\tau}\!_{x}^{\,\epsilon,(\infty)} > t \right)& = \proba\left( \int_{\bcal_{\epsilon}((x,0))} w_{t}^{(\infty)}(z)dz = 0 \right)\\
& = \mathbf{P}\left( \int_{E_{t}^{\epsilon,x}} w_{0}^{(\infty)}(z)dz = 0 \right)\\
& = \mathbf{P}\left(\Vol\left(E_{t}^{\epsilon,x} \cap H \right) = 0\right) \\
& = \mathbf{P}\left(\tilde{\tau}\!_{x}^{\ \epsilon,(\infty)} > t \right).
\end{align*}
Since this result is true for any $t\geq 0$, we obtain the desired equality in distribution.
\end{proof}

For all $x \geq 0$, let $\hp^{x}$ stand for the half-plane
\begin{equation}\label{eqn:def Hx}
\hp^{x} := \left\{(x',y) \in \rde : x' \geq x\right\}.
\end{equation}
In particular, $H=(H^0)^c$ (this apparent contradiction in our notation will lead to simpler notation in the proofs).

In practice, the set of variables $(\tilde{\tau}_{x}^{(\infty)})_{x > 0}$ is not very convenient to use. Indeed, for $x < x'$, even if $\tilde{\tau}_{x}^{(\infty)}$ and $\tilde{\tau}_{x'}^{(\infty)}$ are constructed using the same underlying Poisson point process related events, we do not necessarily have $\tilde{\tau}_{x}^{(\infty)} \leq \tilde{\tau}_{x'}^{(\infty)}$ a.s., as the underlying $\infty$-parent ancestral processes have different starting positions and do not necessarily use the same sequence of events to reach $H$. Therefore, we define another sequence $(\overleftarrow{\tau}\!_{x}^{\,(\infty)})_{x > 0}$ such that for all $x > 0$, $\overleftarrow{\tau}\!_{x}^{\,(\infty)}$ and $\tilde{\tau}_{x}^{(\infty)}$ have the same distribution, but also such that all underlying $\infty$-parent ancestral processes start from the same location $(0,0)$, thereby ensuring that for all $0<x < x'$, we have $\overleftarrow{\tau}_{x}^{\,(\infty)} \leq \overleftarrow{\tau}_{x'}^{\,(\infty)}$ a.s. Moreover, for every $x > 0$, the $\infty$-parent ancestral process associated with $\overleftarrow{\tau}\!_{x}^{\,(\infty)}$ will have the same distribution as the ancestral process associated with $\tilde{\tau}_{x}^{(\infty)}$ reflected with respect to the axis $\{(x/2,y) : y \in \mathbb{R}\}$. In other words, we shall be interested in the expansion of the new process ``from left to right'', starting from $(0,0)$ and reaching abscissa $x$, instead of studying the expansion of  $E^x$ ``from right to left'', starting from $(x,0)$ and reaching abscissa~$0$.

Let $\mu^{\leftarrow}$ be the finite measure on $S_{\mu}$ defined by the property that for all measurable subsets $I,J$ of $(0,+\infty)$ and for all $-\pi/2 < \gamma_{1} \leq \gamma_{2} < \pi/2$,
\begin{equation}\label{eqn:mu back}
\mu^{\leftarrow}(I \times J \times [\gamma_{1},\gamma_{2}]) = \mu(I \times J \times [-\gamma_{2},-\gamma_{1}]).
\end{equation}
Let $(E_{t}^{\epsilon})_{t \geq 0}$, $\epsilon > 0$, be a sequence of $\infty$-parent ancestral processes associated with $\mu^{\leftarrow}$ with initial condition $\bcal_{\epsilon}((0,0))$, constructed using the same underlying Poisson point process $\overline{\Pi}$ (with intensity measure $dt\otimes dz\otimes \mu^{\leftarrow}(da,db,d\gamma)$), and let $(E_{t})_{t \geq 0}$ be the $\infty$-parent ancestral process with initial condition $\{(0,0)\}$ associated with $\mu^{\leftarrow}$, also constructed using the Poisson point process $\overline{\Pi}$. For definiteness, we suppose that these new objects are also constructed on the probability space $(\mathbf{\Omega},\mathbf{\fcal},\bP)$. By the reasoning that led to~\eqref{eqn:lem_2_13}, there exists almost surely $\epsilon_{0} > 0$ such that if $t_{0}$ is the first time at which $(0,0)$ is affected by a reproduction event in $\overline{\Pi}$,
\begin{equation}\label{eqn:new_condition_2_3_prime}
\forall\, \epsilon \in (0, \epsilon_{0}],\qquad E_{t_{0}}^{\epsilon} = E_{t_{0}}\text{ and }\forall\,  t\in [0,t_{0}),\ E_{t}^{\epsilon} = E_{0}^{\epsilon}.
\end{equation}

\begin{definition}\label{defn:dual_tau_x}
For all $x>0$, let
\begin{equation*}
\overleftarrow{\tau}\!_{x}^{\,(\infty)} := \left\lbrace
\begin{array}{ll}
\min \Big\{t \geq 0 :
\lim\limits_{\epsilon \to 0} \Vol\left(
E_{t}^{\epsilon} \cap \hp^{x}
\right) > 0\Big\}  & \mbox{if $\exists\, \epsilon_{0} > 0$ satisfying~\eqref{eqn:new_condition_2_3_prime},}\\
0 & \mbox{otherwise}.
\end{array}
\right.
\end{equation*}
\end{definition}

As for $(\tilde{\tau}\!_{x}^{\,(\infty)})_{x > 0}$, for all $x > 0$, $\overleftarrow{\tau}\!_{x}^{\,(\infty)}$ is well-defined and almost surely equal to
\begin{equation*}
\min\left\{
t \geq 0: E_{t} \cap \hp^{x} \neq \emptyset
\right\}.
\end{equation*}
Let us say that a point $z = (x,y) \in \rde$ is at \textit{horizontal separation} $d$ of the point $z' = (x',y') \in \rde$ if and only if $x - x' = d$. Then informally, $\overleftarrow{\tau}\!_{x}^{\,(\infty)}$ represents the first time the $\infty$-parent ancestral process starting from $(0,0)$ reaches points at horizontal separation of at least $x$ from the starting location. Moreover, we have the following lemma.

\begin{lem}\label{lem:la_suite_est_croissante}
The function $x \to \overleftarrow{\tau}\!_{x}^{\,(\infty)}$ is nondecreasing.
\end{lem}

\begin{proof}
We distinguish two cases. If there does not exist $\epsilon_{0} > 0$ satisfying \eqref{eqn:new_condition_2_3_prime}, then for all $x > 0$, $\overleftarrow{\tau}\!_{x}^{(\infty)} = 0$, and we can conclude.

If there exists $\epsilon_{0} > 0$ satisfying \eqref{eqn:new_condition_2_3_prime}, let $0 < x_{1} < x_{2}$. By similar arguments as in Lemma~\ref{lem:lem1} and its proof, we can write that
\begin{align*}
\forall\, 0 < \epsilon \leq \min(\epsilon_{0},x_{1}/2),\quad \Vol\left(
E_{t}^{\epsilon} \cap \hp^{x_{1}}
\right) &= \Vol\left(
E_{t}^{\min(\epsilon_{0},x_{1}/2)} \cap \hp^{x_{1}}
\right) \\
\text{and } \quad \forall\, 0 < \epsilon \leq \min(\epsilon_{0},x_{2}/2),\quad \Vol\left(
E_{t}^{\epsilon} \cap \hp^{x_{2}}
\right) &= \Vol\left(
E_{t}^{\min(\epsilon_{0},x_{2}/2)}) \cap \hp^{x_{2}}
\right).
\end{align*}
Let us set $\epsilon ' = \min(\epsilon_{0},x_{1}/2,x_{2}/2)$. Then, $\hp^{x_{2}} \subset \hp^{x_{1}}$ and so
\begin{equation*}
\Vol\left(
E_{\overleftarrow{\tau}\!_{x_{2}}^{\,(\infty)}}^{\epsilon '} \cap \hp^{x_{1}}
\right) \geq \Vol\left(
E_{\overleftarrow{\tau}\!_{x_{2}}^{\,(\infty)}}^{\epsilon '} \cap \hp^{x_{2}}
\right) = \lim\limits_{\epsilon \to 0} \Vol\left(
E_{\overleftarrow{\tau}\!_{x_{2}}^{\,(\infty)}}^{\epsilon} \cap \hp^{x_{2}}
\right) > 0
\end{equation*}
by definition of $\overleftarrow{\tau}\!_{x_{2}}^{\,(\infty)}$, and so $\overleftarrow{\tau}\!_{x_{1}}^{\,(\infty)} \leq \overleftarrow{\tau}\!_{x_{2}}^{\,(\infty)}$ and we can conclude.
\end{proof}

Notice that we are now studying the expansion of the backwards-in-time process in the same direction as the occupied area in the forwards-in-time process. On the other hand, $(\tilde{\tau}\!_{x}^{\,(\infty)})_{x > 0}$ corresponds to the expansion of the $\infty$-parent ancestral process in the \textit{opposite direction}. Moreover, we recall that $(\overleftarrow{\tau}\!_{x}^{\,(\infty)})_{x > 0}$ can be seen as being constructed using an underlying $\infty$-parent ancestral process which is the reflection of the one used to construct $(\tilde{\tau}\!_{x}^{\,(\infty)})_{x > 0}$ with respect to the axis $\{(x/2,y) : y \in \mathbb{R}\}$.
This observation yields the following lemma.

\begin{lem}\label{lem:tilde_dual_tau_x}
For all $x > 0$, $\overleftarrow{\tau}\!_{x}^{\,(\infty)}$ and $\tilde{\tau}_{x}^{\,(\infty)}$ have the same distribution.
\end{lem}

\begin{proof}
Let $x > 0$. For all $t \geq 0$ and $\epsilon > 0$, let $\mathrm{Ref}(E_{t}^{\epsilon,x})$ be the reflection of $E_{t}^{\epsilon,x}$ with respect to the axis $\{(x/2,y) : y \in \rmath\}$. Then, $(\mathrm{Ref}(E_{t}^{\epsilon,x}))_{t \geq 0}$, $\epsilon > 0$, is a sequence of $\infty$-parent ancestral processes with initial condition $\bcal_{\epsilon}((0,0))$, all constructed using the same Poisson point process having the same distribution as $\overline{\Pi}$ (that is, whose intensity measure is $dt\otimes dz\otimes \mu^{\leftarrow}(da,db,d\gamma)$), which can be used to construct $\overleftarrow{\tau}\!_{x}^{(\infty)}$. Moreover, for all $t \geq 0$ and $\epsilon > 0$,
\begin{equation*}
\mathrm{Vol}\big(E_{t}^{\epsilon,x} \cap H \big) > 0
\quad \text{if and only if} \quad \mathrm{Vol}\left(
\mathrm{Ref}(E_{t}^{\epsilon,x}) \cap \hp^{x}
\right) > 0,
\end{equation*}
which allows us to conclude.
\end{proof}

The following result is then a direct consequence of Lemmas~\ref{lem:dualite_tau_x} and~\ref{lem:tilde_dual_tau_x}.

\begin{prop}\label{prop:taux_some_distr}
For all $x > 0$, $\overleftarrow{\tau}\!_{x}^{\,(\infty)}$ and $\overrightarrow{\tau}\!_{x}^{\,(\infty)}$ have the same distribution.
\end{prop}
In other words, we can show Theorem~\ref{thm:speed_growth_infty_slfv} by proving a similar result for $\overleftarrow{\tau}\!_{x}^{\,(\infty)}$.

\subsection{A \textit{backwards in time} version of Theorem \ref{thm:speed_growth_k_slfv}}\label{subsec:back k}

We now turn to the case $k < +\infty$, and rephrase the probability of occupancy appearing in Theorem~\ref{thm:speed_growth_k_slfv} in terms of a reflected version of the $k$-parent ancestral process. Recall the definition of $\mu^{\leftarrow}$ given in~\eqref{eqn:mu back}, of the quenched $k$-parent ancestral process given in Definition~\ref{defn:quenched k-parent}, of the half-plane $H^x$ given in \eqref{eqn:def Hx}, and recall the notation $A(\Xi)$ for the set of atoms of the counting measure $\Xi$.
\begin{prop}\label{prop:backwards_k_slfv} Let $k \geq 2$.
Let $\overleftarrow{\Pi}^{+}$ be a Poisson point process on $\mathbb{R}_+\times \rde\times S_\mu\times (\rde)^\infty$ with intensity measure $dt \otimes dz \otimes \mu^{\leftarrow}(da,db,d\gamma)\tilde{u}_{(a,b,\gamma)}(d\mathbf{p})$ defined on the probability space $(\mathbf{\Omega}, \mathbf{\fcal}, \mathbf{P})$. Let $(\overleftarrow{\Xi}_{t}^{(k)})_{t \geq 0}$ be the quenched $k$-parent ancestral process with initial condition $\delta_{(0,0)}$ constructed using $\overleftarrow{\Pi}^{+}$. Then, for all $x > 0$ and $t > 0$,
\begin{equation*}
\proba\left(
\limeps V_{\epsilon}^{-1} \int_{\bcal_{\epsilon}((x,0))} w_{t}^{(k)}(z)dz = 0
\right) = \mathbf{P}\Big(
A\big(
\overleftarrow{\Xi}_{t}^{(k)}
\big) \cap \textit{\text{H}}^{x} = \emptyset
\Big).
\end{equation*}
\end{prop}

\begin{proof}
Let $x > 0$ and $t > 0$. Let $(w_{k,s}^{H})_{s \geq 0}$ be the density of $(M_{s}^{(k)})_{s \geq 0}$ obtained when following the Poisson point process construction described in Section~\ref{subsec:poisson_construction}, and using the Poisson point process $\Pi^{+}$. By Lemma~\ref{lem:property_2_density_slfv}, the value of
\begin{equation*}
\limeps V_{\epsilon}^{-1} \int_{\bcal_{\epsilon}((x,0))} w_{t}^{(k)}(z)dz
\end{equation*}
is almost surely equal to
\begin{equation*}
w_{k,t}^{\text{H}}((x,0)) = 1 - \left(
\prod_{z \in A\big(
\Xi_{k,t}^{\Pi^{+},t,\delta_{(x,0)}}
\big)} \big(1-\mathds{1}_{\text{H}}(z)\big)
\right),
\end{equation*}
where $(\Xi_{k,t'}^{\Pi^{+},t,\delta_{(x,0)}})_{t'\in[0, t]}$ is the quenched $k$-parent ancestral process associated with $\Pi^{+}$, started at time~$t$ and with initial condition $\delta_{(x,0)}$. Therefore,
\begin{align*}
\proba\left(
\limeps V_{\epsilon}^{-1} \int_{\bcal_{\epsilon}((x,0))} w_{t}^{(k)}(z)dz =0
\right) &= \proba\left(
A\Big(
\Xi_{k,t}^{\Pi^{+},t,\delta_{(x,0)}}
\Big) \cap \text{H} = \emptyset
\right) \\
&= \mathbf{P}\left(
A\big(
\overleftarrow{\Xi}_{t}^{(k)}
\big) \cap \text{H}^{x} = \emptyset
\right)
\end{align*}
by symmetry with respect to the axis $\{(x/2,y) : y \in \rmath\}$, which is the desired result.
\end{proof}

\section{Lower bound on the speed of growth of the $\infty$-parent SLFV}\label{sec:lower_bound}
In this section, we show that the growth of the occupied area in the $\infty$-parent SLFV, as well as the growth of its dual, is \textit{at least} linear in time. More precisely, we show the following result.

\begin{prop}\label{prop:lower_bound_growth}
There exists $\nu^{(\infty)} \geq 0$ such that
\begin{equation}\label{limite exp}
\lim\limits_{x \to + \infty} \frac{\bE\big[\overleftarrow{\tau}\!_{x}^{\,(\infty)}\big]}{x} = \nu^{(\infty)}
\end{equation}
and
\begin{equation}\label{limite as}
\lim\limits_{x \to + \infty} \frac{\overleftarrow{\tau}\!_{x}^{\,(\infty)}}{x} = \nu^{(\infty)} \quad \text{a.s.}
\end{equation}
\end{prop}
The proof of this result can be found at the end of Section~\ref{sec:section_2_3}. As $\overleftarrow{\tau}\!_{x}^{\,(\infty)}$ and $\overrightarrow{\tau}\!_{x}^{\,(\infty)}$ are equal in distribution by Proposition \ref{prop:taux_some_distr}, the convergence in \eqref{limite exp} also holds with $\overleftarrow{\tau}\!_{x}^{\,(\infty)}$ replaced by $\overrightarrow{\tau}\!_{x}^{\,(\infty)}$, and the convergence in \eqref{limite as} only holds in distribution (and therefore in probability, as the limit is a constant) when $\overleftarrow{\tau}\!_{x}^{\,(\infty)}$ is replaced by $\overrightarrow{\tau}\!_{x}^{\,(\infty)}$. Indeed, the a.s. convergence stated in \eqref{limite as} exploits the i.i.d. structure of the times needed for $(E_t)_{t\geq 0}$ to reach horizontal separation $n+1$ of $(0,0)$ starting from a single point at horizontal separation $n$ of $(0,0)$. Because there is no analogous subdivision of the reaching time of $(x,0)$ by the occupied area in $M^{(\infty)}$ (as a path of overlapping events reaching $(0,n+1)$ does not necessarily pass by $(0,n)$) and $\overleftarrow{\tau}\!_{x}^{\,(\infty)}$ and $\overrightarrow{\tau}\!_{x}^{\,(\infty)}$ are only equal in distribution, we cannot translate \eqref{limite as} into the a.s. convergence of $\overrightarrow{\tau}\!_{x}^{\,(\infty)}/x$ as $x\rightarrow \infty$.

The key difference with Theorem \ref{thm:speed_growth_infty_slfv} is that Proposition~\ref{prop:lower_bound_growth} does not require $\nu^{(\infty)}$ to be non-zero, and hence only means that the growth is at least linear in time (and faster if $\nu^{(\infty)} = 0$). In Section~\ref{sec:upper_bound}, we show  that indeed we have $\nu^{(\infty)} \neq 0$, or in other words, that the growth is exactly linear in time.

\begin{rem}
Since the limiting speed of growth is given by $(\nu^{(\infty)})^{-1}$, a lower bound on the speed of growth amounts to an \textit{upper bound} on $\lim\limits_{x \to + \infty} x^{-1}\bE\big[\overleftarrow{\tau}\!_{x}^{(\infty)}\big] $.
\end{rem}

In order to prove Proposition \ref{prop:lower_bound_growth}, we first establish a few auxiliary results.
In all that follows, let again $\overline{\Pi}$ be a Poisson point process on $\rmath_{+} \times \rde \times S_{\mu}$ with intensity measure $dt \otimes dz \otimes \mu^{\leftarrow}$, and let $(E_{t})_{t \geq 0}$ be the $\infty$-parent ancestral process with initial condition $\{(0,0)\}$ and associated with $\mu^{\leftarrow}$, constructed using the Poisson point process $\overline{\Pi}$. Here we use the modification of the $\infty$-parent ancestral process introduced earlier, which allows us to take single points as initial conditions.

\subsection{A sub-additivity result}\label{subsec:subadditivity}
We first introduce other $\infty$-parent ancestral processes, coupled to $(E_{t})_{t \geq 0}$ via the Poisson point process $\overline{\Pi}$. For all $z \in \rde$ and $s \geq 0$, let $(E_{t}^{z,s})_{t \geq 0}$ be the $\infty$-parent ancestral process with initial condition $\{z\}$ associated with $\mu^{\leftarrow}$, constructed using \textit{only the reproduction events in $\overline{\Pi}$ occurring strictly after time~$s$}. If $s = 0$, then $(E_{t}^{z,0})_{t \geq 0}$ is the regular $\infty$-parent ancestral process, but if $s \neq 0$, then the process is constant and equal to $\{z\}$ during the time interval $[0,s]$, and only after does it start following the dynamics of an $\infty$-parent ancestral process.

Since all the processes $(E_{t}^{z,s})_{t \geq 0}$, $z \in \rde$, $s \geq 0$ are constructed using the same underlying Poisson point process, we have the following result.

\begin{lem}\label{lem:inclusion_ancestral_process}
For all $z_{1}, z_{2} \in \rde$ and $0 < s_{1} < s_{2}$, if $z_{2} \in E_{s_{2}}^{z_{1},s_{1}}$, then a.s. for all $t \geq s_{2}$, we have
\begin{equation*}
E_{t}^{z_{2},s_{2}} \subseteq E_{t}^{z_{1},s_{1}}.
\end{equation*}
\end{lem}

We now introduce the following family of random variables. First, for all $n \in \nmath$, let
\begin{align*}
T_{0,n} :&= \min\left\{t \geq 0 : E_{t}^{(0,0),0} \cap \hp^{4n \rcal_{\mu}} \neq \emptyset\right\} = \min\left\{t \geq 0 : E_{t} \cap \hp^{4n \rcal_{\mu}} \neq \emptyset\right\} = \overleftarrow{\tau}\!_{4n \rcal_{\mu}}^{\,(\infty)} \quad \text{a.s.,}
\end{align*}
and let $P_{n}$ be sampled uniformly at random among the points in $E_{T_{0,n}}$ at horizontal separation of \textit{exactly} $4n \rcal_{\mu}$ from $(0,0)$. That is, $P_{n}$ is a uniform sample from the compact set
\begin{equation*}
E_{T_{0,n}} \cap \left\{
(x,y) \in \rde : x = 4n \rcal_{\mu}
\right\}.
\end{equation*}
Moreover, for all $0 \leq m \leq n \in \nmath$, let
\begin{equation*}
T_{m,n} := \min\left\{
t \geq 0 : E_{t+T_{0,m}}^{P_{m},T_{0,m}} \cap \hp^{4n \rcal_{\mu}} \neq \emptyset
\right\},
\end{equation*}
where $E_{t + T_{0,m}}^{P_{m},T_{0,m}}$ corresponds to the $\infty$-parent ancestral process started from $P_{m}$ at time $T_{0,m}$.

By construction, the family $(T_{m,n})_{0 \leq m \leq n}$ satisfies the following lemma.
\begin{lem}\label{lem:liggett_1} For all $0 < m < n \in \nmath$, we have
\begin{equation*}
T_{0,n} \leq T_{0,m} + T_{m,n}.
\end{equation*}
\end{lem}

\begin{proof}
Let $0 < m < n \in \nmath$. By definition,
\begin{equation*}
P_{m} \in E_{T_{0,m}} = E_{T_{0,m}}^{(0,0),0}.
\end{equation*}
Therefore, using Lemma \ref{lem:inclusion_ancestral_process}, we have a.s. for all $t \geq 0$
\begin{equation*}
E_{t + T_{0,m}}^{P_{m},T_{0,m}} \subseteq E_{t+  T_{0,m}}.
\end{equation*}
In particular, this is true for $t = T_{m,n}$, hence
\begin{equation*}
E_{T_{m,n}+T_{0,m}}^{P_{m},T_{0,m}} \cap \hp^{4n \rcal_{\mu}} \subseteq
E_{T_{m,n}+T_{0,m}} \cap \hp^{4n \rcal_{\mu}},
\end{equation*}
from which we deduce (by definition of $T_{0,m}$, $P_{m}$ and $T_{m,n}$) that
\begin{equation*}
E_{T_{m,n}+T_{0,m}} \cap \hp^{4n \rcal_{\mu}} \neq \emptyset.
\end{equation*}
Therefore,
\begin{align*}
T_{0,n} &= \min \left\{t > 0 : E_{t} \cap \hp^{4n \rcal_{\mu}} \neq \emptyset \right\} \leq T_{m,n} + T_{0,m}.
\end{align*}
\end{proof}

By invariance by translation of the underlying Poisson point processes, the following lemmas also hold true.
\begin{lem}\label{lem:liggett_2}
For all $n \in \nmath$, the joint distribution of $(T_{n+1,n+l+1})_{l \geq 1}$ is the same as that of $(T_{n,n+l})_{l \geq 1}$.
\end{lem}

\begin{lem}\label{lem:liggett_3}
For each $l \in \nmath \setminus \{0\}$, the random variables $(T_{nl,(n+1)l})_{n \geq 1}$ are i.i.d. In particular, the sequence $(T_{nl,(n+1)l})_{n \geq 1}$ is ergodic.
\end{lem}

In order to use Theorem 1.10 from \cite{liggett1985improved} and conclude, all we need is to show that the following lemma is true.

\begin{lem}\label{lem:lem_liggett_4}
For all $n \in \nmath$, we have $\bE[T_{0,n}] < + \infty$.
\end{lem}

Lemma \ref{lem:lem_liggett_4} is proved at the end of Section~\ref{sec:section_2_3}, using the so-called \textit{express chain} that we now define. We shall then use it to complete the proof of Proposition~\ref{prop:lower_bound_growth} (again at the end of Section~\ref{sec:section_2_3}).

\subsection{Definition of the express chain}\label{subsec:express chain}
The use of what we shall call the \textit{express chain} can be motivated by the following observation. In each ellipse with centre $z = (x,y) \in \rde$ and parameters $(a,b,\gamma)$, there exists exactly one point for which the horizontal separation from the centre is maximal: the point with coordinates
\begin{equation*}
(x + a\cos(\theta_{max})\cos(\gamma)-b\sin(\theta_{\max})\sin(\gamma),
y + a\cos(\theta_{max})\sin(\gamma)+b\sin(\theta_{max})\cos(\gamma)),
\end{equation*}
where $\theta_{max} = \arctan(-b a^{-1}\tan(\gamma))$. This point is at horizontal separation
\[\dhor\]
from $z$ (see Figure~\ref{fig:dessin_ellipse}).
\begin{figure}[t]
\centering
\includegraphics[width = 0.5\linewidth]{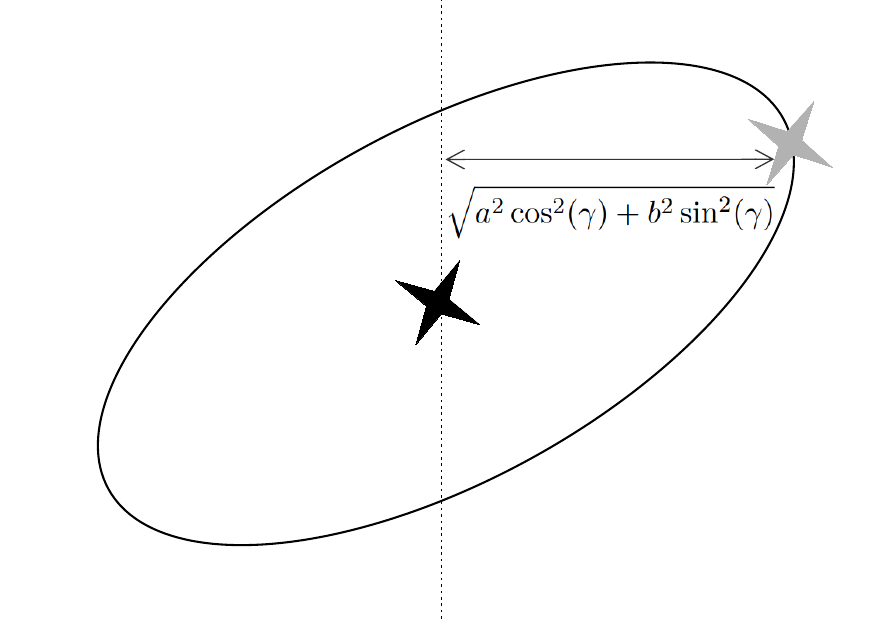}
\caption{Ellipse with parameters $(a,b,\gamma)$.}\label{fig:dessin_ellipse}
\end{figure}

If we take this potential parent, wait until it is affected by a new reproduction event, and repeat, we obtain a Markov jump process, jumping at rate
\begin{equation}\label{jump rate}
\Jm:= \int_{S_{\mu}} V_{a,b}\, \mu(da,db,d\gamma),
\end{equation}
and going away from $0$ at an average \textit{horizontal speed} of $\Jm\bE\left[\dhor\right]$ (modulo some stochasticity due to the location of the centre of the reproduction event). See Appendix \ref{appendix:geom_ellipse} for the proof of the geometrical properties of ellipses used throughout this section.

Formally, the express chain is defined as follows.

\begin{definition}
The express chain associated with $(E_{t})_{t \geq 0}$ (constructed using $\overline{\Pi}$), denoted $(C_{t}^{express})_{t \geq 0}$, is the $\rde$-valued Markov process defined as follows.

First, we set $C_{0}^{express} = (0,0)$. Then, for all $(t,z_{c},a,b,\gamma) \in \overline{\Pi}$, if $C_{t-}^{express} \in \bfrak_{a,b,\gamma}(z_{c})$ and $z_{c} = (x_{c},y_{c})$, we set:
\begin{equation*}
C_{t}^{express} = \left(X_{t}^{express}, Y_{t}^{express} \right)
\end{equation*}
where
\begin{align*}
X_{t}^{express} &= x_{c} + a\cos(\theta_{max})\cos(\gamma)-b\sin(\theta_{\max})\sin(\gamma), \\
Y_{t}^{express} &= y_{c} + a\cos(\theta_{max})\sin(\gamma)+b\sin(\theta_{max})\cos(\gamma)
\end{align*}
and $\theta_{max} = \arctan(-b a^{-1} \tan(\gamma))$.
\end{definition}
See Figure~\ref{fig:dessin_express_chain} for an illustration of how to construct the express chain.

\begin{figure}[t]
\centering
\begin{subfigure}[b]{0.45\textwidth}
\includegraphics[width = \linewidth]{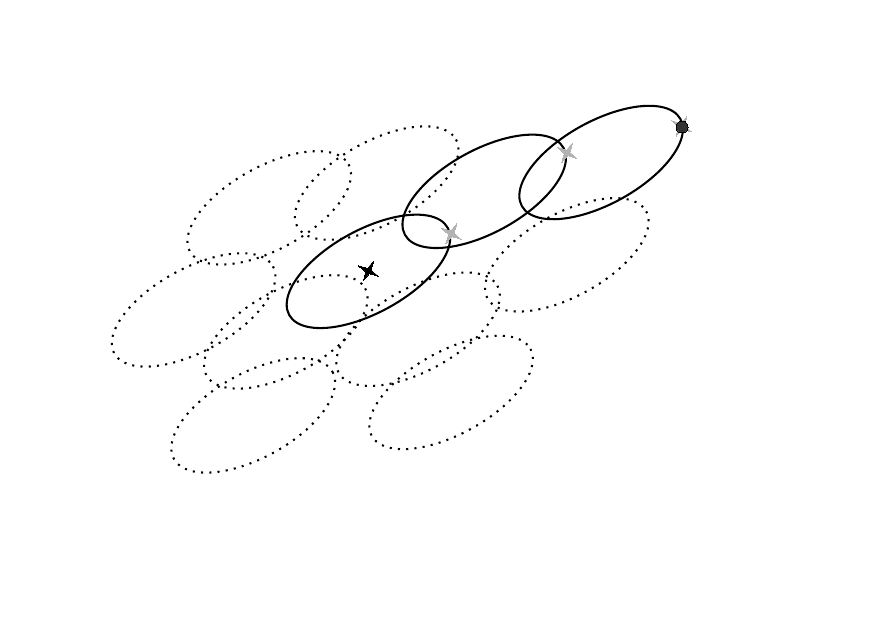}
\end{subfigure}
\begin{subfigure}[b]{0.45\textwidth}
\includegraphics[width = \linewidth]{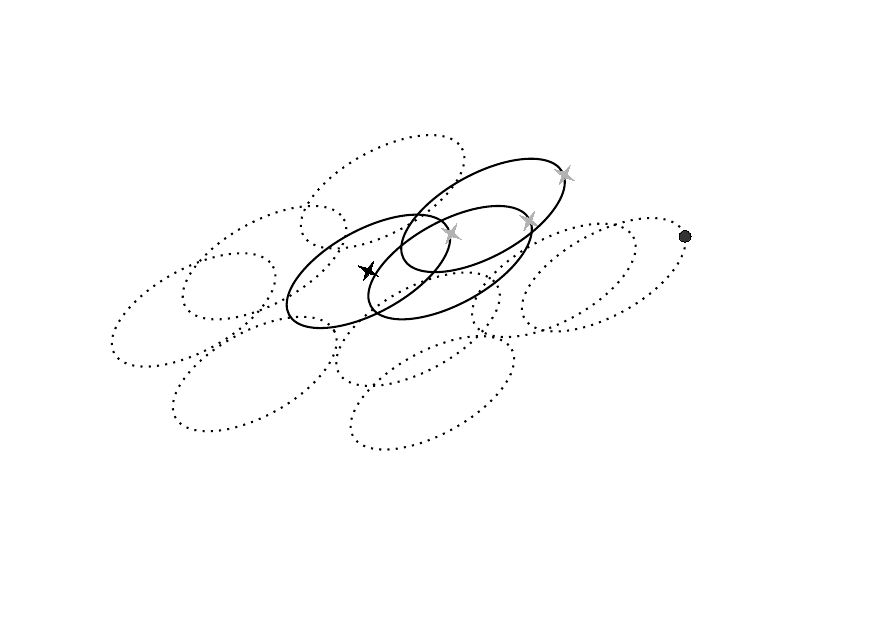}
\end{subfigure}
\caption{Comparison between the express chain and the $\infty$-parent ancestral process at time $t$, for two realisations of the $\infty$-parent ancestral process. The ellipses indicate the reproduction events which affected the $\infty$-parent ancestral process until time $t$. The crosses represent the successive positions of the express chain, and the black point indicates the location reached by the $\infty$-parent ancestral process at time $t$ which maximises the horizontal separation from the starting location. This location coincides with the one reached by the express chain in the first case, but not in the second case.
}\label{fig:dessin_express_chain}
\end{figure}

The interest of the express chain lies in the following observation, whose proof is a direct consequence of the definition of $T_{0,n}$.
\begin{lem}\label{lem:interet_express_chain}
Let $n \in \nmath\setminus\{0\}$. For all $t \geq 0$, we have $\{C_{t}^{express} \in \hp^{4n \rcal_{\mu}}\}\subset \{T_{0,n} \leq t\}$.
\end{lem}

\begin{proof}
Let $t \geq 0$. Since $C_{t}^{express}\in E_t$ by construction, the fact the $C_{t}^{express} \in \hp^{4n \rcal_{\mu}}$ implies that $E_t\cap \hp^{4n \rcal_{\mu}}\neq \emptyset$, and thus that $T_{0,n}\leq t$.
\end{proof}

Therefore, if for all $n \in \nmath$, we set
\begin{equation*}
T_{0,n}^{express} := \min\left\{ t \geq 0 : C_{t}^{express} \in \hp^{4n \rcal_{\mu}}\right\},
\end{equation*}
then for all $n \in \nmath$ we have
\begin{equation}\label{eqn:esp_comparaison_express_chain}
T_{0,n} \leq T_{0,n}^{express}\quad \text{a.s.}, \qquad \text{and so} \quad \bE[T_{0,n}]\leq \bE\big[T_{0,n}^{express}\big].
\end{equation}
In other words, in order to show Lemma~\ref{lem:lem_liggett_4}, it is sufficient to obtain a similar result on $\bE[T_{0,n}^{express}]$.

Before studying the properties of the express chain in the next section, we introduce the following notation. For all $t \geq 0$, let $N_{t}^{express}$ be the number of jumps of the express chain in the time interval $[0,t]$. For all $i \in \nmath \setminus \{0\}$, let $t_{i}$ be the instant of the $i$-th jump of the express chain, let $R_{i} = (R_{i}^{X},R_{i}^{Y})$ be the coordinates of the centre of the reproduction event triggering this jump, and let $(a_{i},b_{i},\gamma_{i})$ be the parameters of the ellipse affected by the reproduction event. We then set for all $i\in \nmath\setminus\{0\}$
\begin{equation*}
D_{i} = \sqrt{a_{i}^{2} \cos^{2}(\gamma_{i})+b_{i}^{2}\sin^{2}(\gamma_{i})}.
\end{equation*}
The random variable $D_i$ thus encodes the distance between the centre of the reproduction event and the right-most point in the corresponding ellipse. Note that the random variables $(D_{i})_{i\geq 1}$ are i.i.d.

\subsection{Properties of the express chain}\label{sec:section_2_3}
We now study some properties of the express chain, in order to obtain an upper bound on $\bE[T_{0,n}^{express}]$.

Let $t \geq 0$. By construction,
\begin{align}
X_{t}^{express} &= \sum_{i = 1}^{N_{t}^{express}}\left(
R_{i}^{X} - X_{t_{i}-}^{express} + D_{i}
\right) \label{eqn:express_chain} \\
&= \sum_{i = 1}^{N_{t}^{express}}\Big(R_{i}^{X} - X_{t_{i}-}^{express} + D_{i} - \bE[D_{1}]\Big) + \bE[D_{1}]N_{t}^{express}. \nonumber
\end{align}
The random variables $(R_{i}^{X}-X_{t_{i}-}^{express}+D_{i}-\bE[D_{i}])_{i \geq 1}$ are i.i.d, bounded and with mean $0$ (see Appendix~\ref{appendix:geom_ellipse}). We can then apply Hoeffding's inequality \cite{hoeffding1963probability} and obtain the following lemma.

\begin{lem}\label{lem:pos_express_chain_conditional}
There exists $C_{1} > 0$ such that for all $t \geq 0$ and $n \in \nmath \setminus \{0\}$, for all
$k > 4n \rcal_{\mu}\tabg^{-1}$, we have
\begin{align*}
&\bP\left(\xte < 4n \rcal_{\mu} \Big|\nte = k\right) \leq \exp\left(- \frac{\left( \tabg k - 4n \rcal_{\mu}\right)^{2}}{C_{1}k}\right).
\end{align*}
Consequently, for all $C' > 4n \rcal_{\mu}\tabg^{-1}$,
\begin{equation*}
\bP\left(\xte < 4n \rcal_{\mu}\Big|\nte > C'\right)\leq \exp\left(- \frac{\left(C' \tabg - 4n \rcal_{\mu}\right)^{2}}{C_{1}C'} \right).
\end{equation*}
\end{lem}

\begin{proof}
First, notice that the second part of the lemma is a direct consequence of the first part, along with the variations of the function $x \to \exp(-(xc-d)^{2} x^{-1})$, $c,d > 0$. As concerns the first part of the lemma, let $t \geq 0$, $n \in \nmath\setminus \{0\}$ and $k > 4n \rcal_{\mu}\tabg^{-1}$. Using (\ref{eqn:express_chain}), we have
\begin{align*}
&\bP\left(\xte < 4n\rcal_{\mu}\Big|\nte = k\right) \\
&\leq \bP\left(\sum_{i = 1}^{k} (R_{i}^{X}-X_{t_{i}-}^{express} + D_{i} - \bE[D_{1}]) < 4n \rcal_{\mu} - \tabg k\Big|\nte = k\right) \\
&= \bP\left(\sum_{i = 1}^{k} (-R_{i}^{X}+X_{t_{i}-}^{express} - D_{i} + \bE[D_{1}]) > \tabg k -4n \rcal_{\mu}\Big|\nte = k\right).
\end{align*}
We conclude by using Hoeffding's inequality together with the fact that for all $i \geq 1$,
\begin{align*}
\left|R_{i}^{X}-X_{t_{i}-}^{express}\right| &< 4\rcal_{\mu} \qquad \text{and} \qquad  \left|
D_{i}-\bE[D_{1}] \right| \leq 2\rcal_{\mu}.
\end{align*}
\end{proof}

In order to bound $\bE[T_{0,n}^{express}]$, we also need to control the number of jumps made by the express chain over the time interval $[0,t]$. Recall the definition of $\Jm$ given in \eqref{jump rate}.
\begin{lem}\label{lem:control_jumps}
There exists $C_{\otimes} > 0$ such that for all $t > 0$,
\begin{equation*}
\bP\left(\nte \leq 0.1 \Jm t\right) \leq \exp(-C_{\otimes}\Jm t).
\end{equation*}
\end{lem}

\begin{proof}
The proof relies on the fact that the express chain jumps at rate $\Jm$. Hence, $N_{t}^{express}$ is Poisson distributed with parameter $\Jm t$.

Let $t > 0$. Using a standard Chernoff bound, we obtain
\begin{align*}
\bP\left(
N_{t}^{express} \leq 0.1\Jm t \right) &\leq \frac{(e\Jm t)^{0.1\Jm t}e^{-\Jm t}}{(0.1\Jm t)^{0.1\Jm t}} \\
&= \frac{\exp(0.1\Jm t)\exp(-\Jm t)}{\exp(0.1\Jm t\,\ln(0.1))} \\
&= \exp(0.1\Jm t- \Jm t-0.1\Jm t \ln(0.1)),
\end{align*}
and so taking $C_{\otimes} = \Jm(0.9 - 0.1 \ln(10))$ allows us to conclude.
\end{proof}

Combining Lemmas \ref{lem:pos_express_chain_conditional} and \ref{lem:control_jumps}, we obtain an upper bound for $\bE[T_{0,n}^{express}]$.
\begin{lem}\label{lem:upper_bound_front}
There exists $C_{2}>0$ such that for all $n \in \nmath$,
\begin{equation*}
\bE\Big[T_{0,n}^{express}\Big] \leq C_{2} n.
\end{equation*}
\end{lem}

\begin{proof}
Let $n \in \nmath$. Then, using Lemmas \ref{lem:pos_express_chain_conditional} and \ref{lem:control_jumps} to pass from the fourth to the fifth line below, we have
\begin{align}
\bE\big[T_{0,n}^{express}\big] = & \int_{0}^{\infty} \bP\left(T_{0,n}^{express} > t\right)dt \nonumber \\
=& \int_{0}^{\infty} \bP\left(X_{t}^{express} < 4n \rcal_{\mu}\right)dt \nonumber\\
\leq & \frac{100n \rcal_{\mu}}{\Jm\tabg} + \int_{\frac{100n  \rcal_{\mu}}{\Jm\tabg}}^{\infty} \bP\left(X_{t}^{express} < 4n\rcal_{\mu} \right)dt \nonumber\\
\leq & \frac{100n  \rcal_{\mu}}{\Jm\tabg} + \int_{\frac{100n \rcal_{\mu}}{\Jm\tabg}}^{\infty} \bP\left(N_{t}^{express} \leq \frac{\Jm t}{10} \right)dt \nonumber\\
&+ \int_{\frac{100n \rcal_{\mu}}{\Jm \tabg}}^{\infty} \bP\left(\left\{\xte < 4n   \rcal_{\mu}
\right\} \cap \left\{ \nte > \frac{\Jm t}{10}\right\}\right)dt \nonumber\\
\leq & \frac{100n \rcal_{\mu}}{\Jm \tabg} + \int_{\frac{100n \rcal_{\mu}}{\Jm\tabg}}^{\infty} \exp(-C_{\otimes} \Jm t)dt \label{eqn:dans_align}\\
&+ \int_{\frac{100n \rcal_{\mu}}{\Jm \tabg}}^{\infty}\exp\left(
-\frac{10}{C_{1}\Jm t} \left(\frac{\Jm t \tabg}{10} - 4 n \rcal_{\mu}
\right)^{2}\right)dt. \nonumber
\end{align}
Moreover, we have
\begin{align*}
 &\int_{\frac{100n \rcal_{\mu}}{\Jm\tabg}}^{\infty}\exp\bigg(-\frac{10}{C_{1}\Jm t}\left(\frac{\Jm t \tabg}{10} - 4 n   \rcal_{\mu}\right)^{2}\bigg)dt \\
&= \exp\left(\frac{2\tabg}{C_1}   4n   \rcal_{\mu}\right)\int_{\frac{100n  \rcal_{\mu}}{\Jm\tabg}}^{\infty} \exp\left(-\frac{10}{C_{1}\Jm t}  \frac{t^{2} \Jm^2\tabg^{2}}{100}\right)   \exp\left(-\frac{10}{C_{1}\Jm t}  16n^{2}\rcal_{\mu}^{2}\right)dt \\
&\leq \exp\left(\frac{8n}{C_{1}}  \rcal_{\mu} \tabg\right)  \int_{\frac{100n  \rcal_{\mu}}{\Jm\tabg}}^{\infty} \exp\left(-\frac{\Jm t}{10C_{1}}  \tabg^{2}\right)dt \\
&= \exp\left(\frac{8n}{C_{1}}  \rcal_{\mu} \tabg\right)\frac{10C_{1}}{\Jm \tabg^{2}}   \exp\left(
-\frac{\Jm \tabg^{2}}{10C_{1}}  \frac{100n  \rcal_{\mu}}{\Jm\tabg}\right) \\
&= \frac{10C_{1}}{\Jm \tabg^{2}}  \exp\left(-\frac{2n}{C_{1}}  \rcal_{\mu} \tabg\right).
\end{align*}
Since the first term on the r.h.s of (\ref{eqn:dans_align}) is proportional to $n$, and the second and third terms decrease exponentially fast in $n$, we can conclude.
\end{proof}

We can now conclude the proof of Lemma \ref{lem:lem_liggett_4}.
\begin{proof}(Lemma \ref{lem:lem_liggett_4}) Let $n \in \nmath$. By Lemma \ref{lem:upper_bound_front},
\begin{equation*}
\bE\big[T_{0,n}^{express}\big] < + \infty,
\end{equation*}
and so by (\ref{eqn:esp_comparaison_express_chain}),
\begin{equation*}
\bE\big[T_{0,n}\big] < + \infty.
\end{equation*}
\end{proof}

We conclude this section with the proof of Proposition \ref{prop:lower_bound_growth}.
\begin{proof}
(Proposition \ref{prop:lower_bound_growth})
Combining Lemmas \ref{lem:upper_bound_front}, \ref{lem:liggett_1}, \ref{lem:liggett_2} and \ref{lem:liggett_3} with the fact that $T_{0,n}\geq 0$ for all~$n$, we obtain that the family $(T_{m,n})_{0 \leq m \leq n}$ satisfies all the assumptions (1.3), (1.7), (1.8) and (1.9) of Theorem 1.10 from \cite{liggett1985improved}.
Therefore, there exists $\nu^{(\infty)} \geq 0$ such that
\begin{equation*}
\lim\limits_{n \to + \infty} \frac{\bE\big[T_{0,n}\big]}{4n\rcal_{\mu}} =
\lim\limits_{n \to + \infty} \frac{\bE\big[\overleftarrow{\tau}\!_{4n \rcal_{\mu}}^{\,(\infty)}\big]}{4n\rcal_{\mu}} = \nu^{(\infty)}.
\end{equation*}
Moreover, using the ergodicity property stated in Lemma~\ref{lem:liggett_3}, we also have
\begin{equation*}
\lim\limits_{n \to + \infty} \frac{T_{0,n}}{4n \rcal_{\mu}} =
\lim\limits_{n \to + \infty} \frac{\overleftarrow{\tau}\!_{4n \rcal_{\mu}}^{\,(\infty)}}{4n \rcal_{\mu}} = \nu^{(\infty)} \quad \text{ a.s.}
\end{equation*}
(We have multiplied $n$ by the constant $4\rcal_{\mu}$ in the denominator so that the limit $\nu^{(\infty)}$ coincides with the limit announced in Theorem~\ref{thm:speed_growth_infty_slfv}). We conclude by standard upper and lower bounding arguments, using the fact that $x \to \bE[\overleftarrow{\tau}\!_{x}^{\,(\infty)}]$ and $x \to \overleftarrow{\tau}\!_{x}^{\,(\infty)}$ are nondecreasing functions (which is a consequence of Lemma~\ref{lem:la_suite_est_croissante}).
\end{proof}

\begin{rem}\label{rem:speed in example}
If we see $(X_{t}^{express})_{t \geq 0}$ as a cumulative process (in the sense of Chapter~VI.3 in \cite{asmussen2008applied}), it is possible to use Theorem~3.1 from this chapter and show that the limiting horizontal speed of advance of the express chain is equal to $\Jm \tabg$, yielding an explicit lower bound on the speed of growth of the occupied area in the $\infty$-parent SLFV. In particular, if all reproduction ellipses have the same shape parameters $(a,b,\gamma)$ and if the total mass of $\mu$ is $V_{a,b}^{-1}$, which corresponds to the case investigated using numerical simulations in Section~\ref{sec:numerical_simulations}, then the lower bound on the speed of growth is given by $\sqrt{a^{2}\cos^{2}(\gamma)+b^{2}\sin^{2}(\gamma)}$.
\end{rem}

\section{Upper bound on the speed of growth of the $\infty$-parent SLFV}\label{sec:upper_bound}
Recall that $(E_{t})_{t \geq 0}$ is the $\infty$-parent ancestral process with initial condition $\{(0,0)\}$ associated with $\mu^{\leftarrow}$, constructed using $\overline{\Pi}$.

In this section, we complete the result proved in Section~\ref{sec:lower_bound} by showing that the growth of the $\infty$-parent SLFV and of its dual counterpart are \textit{at most} linear in time. This can be rewritten as a limiting property of $\bE[\overleftarrow{\tau}\!_{x}^{\,(\infty)}]$ as follows.

\begin{prop}\label{prop:upper_bound_growth}
\begin{equation*}
\lim\limits_{x \to + \infty}\frac{\bE\big[\overleftarrow{\tau}\!_{x}^{\,(\infty)}\big]}{x} > 0.
\end{equation*}
\end{prop}
The proof can be found at the end of Section \ref{sec:section_4_3}. Combining this result with Propositions~\ref{prop:taux_some_distr} and \ref{prop:lower_bound_growth} gives us the result of Theorem~\ref{thm:speed_growth_infty_slfv}.

To show the result of Proposition~\ref{prop:upper_bound_growth}, we first observe that it is sufficient to focus on the case in which all reproduction events are balls of fixed radius. Indeed, we have the following comparison result.

\begin{lem}\label{lem:lem2} Let $\overline{\Pi}^{\rcal_{\mu}}$ be a Poisson point process on $\mathbb{R}_{+} \times \mathbb{R}^{2} \times S_{\mu}$ with intensity measure
\begin{equation*}
\mu^{\leftarrow}(S_{\mu}) dt \otimes dz \otimes \delta_{\rcal_{\mu}}(da) \otimes \delta_{\rcal_{\mu}}(db) \otimes \delta_{0}(\gamma).
\end{equation*}
Let $(E_{t}^{\rcal_{\mu}})_{t \geq 0}$ be the $\infty$-parent ancestral process with initial condition $\{(0,0)\}$ constructed using $\overline{\Pi}^{\rcal_{\mu}}$, and for all $x > 0$, let $\overleftarrow{\tau}\!_x^{\,\rcal_{\mu}}$ be the first time $(E_{t}^{\rcal_{\mu}})_{t \geq 0}$ reaches $\hp^{x}$, in the sense of Definition~\ref{defn:dual_tau_x}.
Then, for all $x > 0$,
\begin{equation*}
\bE\big[\overleftarrow{\tau}\!_{x}^{\,(\infty)}\big] \geq \bE\left[
\overleftarrow{\tau}\!_{x}^{\,\rcal_{\mu}}
\right].
\end{equation*}
\end{lem}

\begin{proof}
The proof relies on the following coupling between $(E_{t})_{t \geq 0}$ and $(E_{t}^{\rcal_{\mu}})_{t \geq 0}$. Instead of being independent from $\overline{\Pi}$, the Poisson point process $\overline{\Pi}^{\rcal_{\mu}}$ is constructed using the points from $\overline{\Pi}$, as follows:
\[\overline{\Pi}^{\rcal_{\mu}} := \big\{(t,z,\rcal_{\mu},\rcal_{\mu},0):\ (t,z,a,b,\gamma)\in \overline{\Pi}\big\}.\]
That is, we replace the elliptical area of each reproduction event in $\overline{\Pi}$ by a ball of maximal radius $\rcal_{\mu}$. Since $\bfrak_{a,b,\gamma}(z) \subseteq \bcal_{\rcal_{\mu}}(z)$, this coupling ensures that
\begin{equation*}
\forall\, t \geq 0, E_{t} \subseteq E_{t}^{\rcal_{\mu}} \text{ a.s.}
\end{equation*}
Therefore, for all $x > 0$,
\begin{equation*}
\min\left\{
t \geq 0 : E_{t} \cap \hp^{x} \neq \emptyset
\right\} \geq \min\left\{
t \geq 0 : E_{t}^{\rcal_{\mu}} \cap \hp^{x}
\right\} \text{ a.s.},
\end{equation*}
which allows us to conclude.
\end{proof}
To alleviate the notation, we only provide the proof of Proposition~\ref{prop:upper_bound_growth} when $\mu^{\leftarrow}=\pi^{-1} \delta_1\otimes \delta_1\otimes \delta_0$, so that all reproduction events happen in balls of radius $1$ and the rate at which a given point in space is overlapped by an event is equal to $1$. The proof can be easily generalised to balls of arbitrary fixed radius and any intensity of events by a simple scaling of time and space. We can then use Lemma~\ref{lem:lem2} to obtain the corresponding result for ellipses with bounded parameters.

\subsection{A first-passage percolation problem}\label{subsec:percolation}
We consider the graph $\gcal$ on the vertex set $\zmath^{2}$, in which $(i,j)$ and $(i',j')$ are connected by an edge if and only if
\begin{align*}
(i',j') \in \big\{(i+1,j),(i-1,j),(i,j+1),(i,j-1),&(i-1,j-1),(i-1,j+1),\\
& \qquad (i+1,j-1),(i+1,j+1)\big\}.
\end{align*}

To each edge $e$ of $\gcal$, we associate an independent exponential random variable
\begin{equation*}
\ecal_{e} \sim Exp(16\times \pi^{-1}).
\end{equation*}
$\ecal_{e}$ corresponds to the time needed to pass through the corresponding edge. Following standard terminology in first-passage percolation, we call it the \textit{passage time} of the edge. The choice of the rate of the exponential distribution ensures we can later compare the growth of the $\infty$-parent ancestral process to the first-passage percolation problem we now introduce.

If $\Gamma$ is a (potentially infinite) path formed by the edges $e_{1},... , e_{n},...$, then the passage time of the path $\Gamma$ is defined as
\begin{equation*}
\ecal_{\Gamma} = \sum_{e \in \Gamma} \ecal_{e}.
\end{equation*}
If $z_{1}, z_{2} \in \zmath^{2}$, we define the first-passage time $\ecal_{z_{1},z_{2}}$ from $z_{1}$ to $z_{2}$ (or equivalently, from $z_{2}$ to $z_{1}$ since $\gcal$ is not oriented) as the minimum over the passage times of all the (finite) paths going from $z_{1}$ to $z_{2}$. We then define
\begin{equation*}
\overleftarrow{\tau}\!_{n}^{fpp} := \min\left\{
t \geq 0: \exists m \in \zmath, \ecal_{(0,0),(n,m)} \leq t
\right\}.
\end{equation*}
In other words, $\overleftarrow{\tau}\!_{n}^{fpp}$ is the time needed to reach a point at horizontal separation $n$ from the origin, starting from the origin.

The interest of this first-passage percolation problem lies in the fact that since the passage time of any given edge is a.s. strictly positive, generalising Theorem~6.7 from \cite{smythe1978first}
to our lattice yields the following result.

\begin{lem}\label{lem:cvg_esp_fpp} We have
\begin{equation*}
\lim\limits_{n \to + \infty} \frac{\bE\big[\overleftarrow{\tau}\!_{n}^{fpp}\big]}{n} > 0.
\end{equation*}
\end{lem}

In order to use this lemma and show Proposition~\ref{prop:upper_bound_growth}, we need to be able to compare $\overleftarrow{\tau}\!_{x}^{\,(\infty)}$ and $\overleftarrow{\tau}\!_{n}^{\,fpp}$. The main obstacle to this comparison lies in the fact that the $\infty$-parent ancestral process is continuous in space, while the first-passage percolation problem is defined on a graph. Therefore, we now introduce a way to discretise the $\infty$-parent ancestral process.

\subsection{Discretisation of the $\infty$-parent ancestral process}\label{subsec:discretisation}
In order to discretise the $\infty$-parent ancestral process, we first place a grid on $\rmath^{2}$, and associate a cell to each site of the lattice. Let
\begin{equation*}
\vcal := \left\{
(4i,4j) : (i,j) \in \zmath^{2}
\right\}
\end{equation*}
be the underlying grid, and for all $(i,j) \in \zmath^{2}$, let $\ccal_{i,j}$ be the square with centre $(4i,4j)$ and side length~$2$.
That is,
\begin{equation*}
\ccal_{i,j} := \left\{
(x,y) \in \mathbb{R}^{2} : |x - 4i| \leq 1 \text{ and } |y-4j| \leq 1
\right\}.
\end{equation*}
Each $\ccal_{i,j}$, $i,j \in \zmath$ corresponds to the cell associated with the site $(4i,4j)$ of the grid $\vcal$. See Figure~\ref{fig:lattice_ball} for an illustration.
\begin{figure}[t]
\centering
\includegraphics[width = 0.7\linewidth]{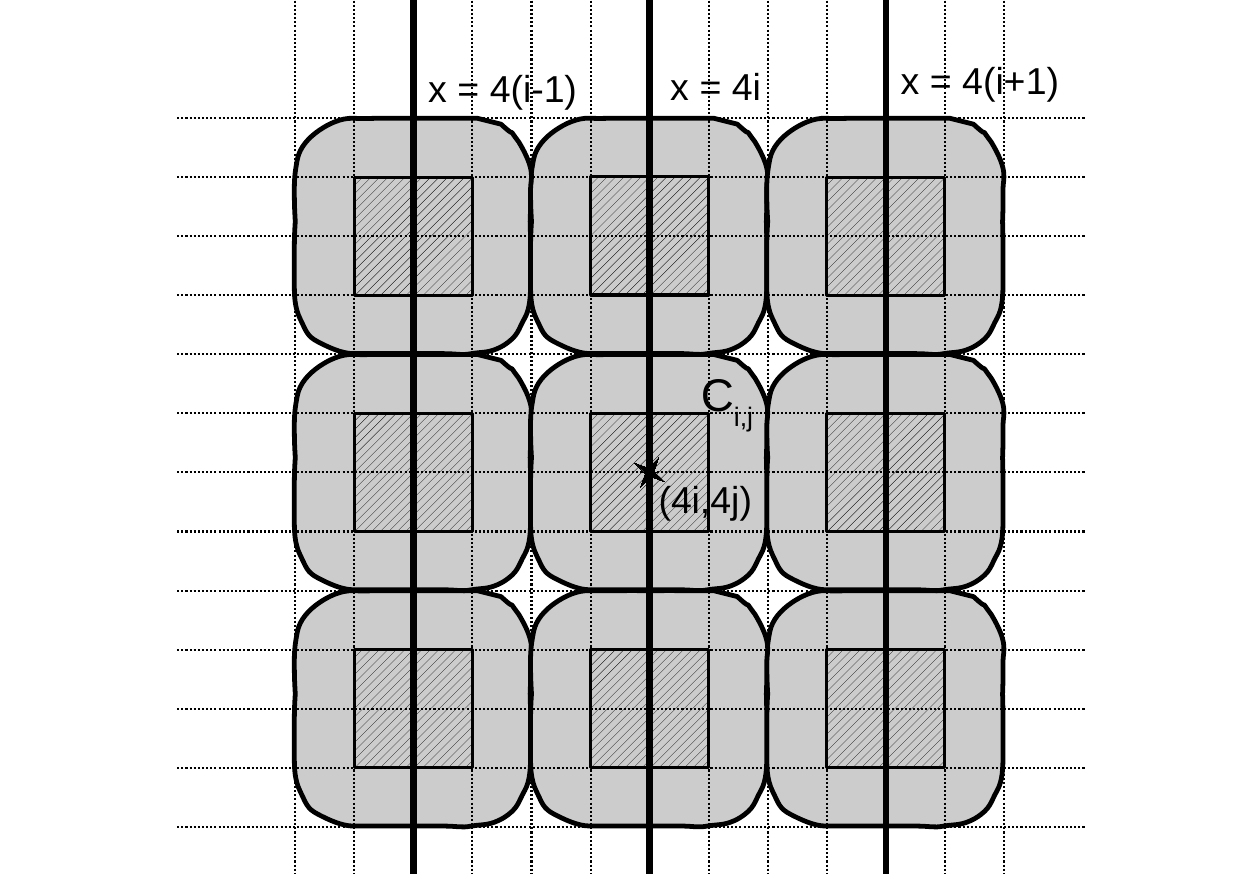}
\caption{Grid used to discretise the $\infty$-parent ancestral process when reproduction events are balls. Each hatched square corresponds to a cell. The grey area around the site $(4i,4j)$, indicated by a black cross, corresponds to the area in which centres of reproduction events have to fall to intersect the cell $\ccal_{i,j}$. The white areas contain all the centres of reproduction events which do not intersect any cell.}\label{fig:lattice_ball}
\end{figure}

This construction satisfies the two following key properties, that are consequences of the fact that all reproduction events are balls of radius $1$.
\begin{enumerate}
\item For all $z = (x,y) \in \rde$, unless $x = 4i+2$, $i \in \mathbb{Z}$ or $y = 4j+2$, $j \in \mathbb{Z}$, the ball $\bcal_{1}(z)$ intersects at most one cell.

In other words, each reproduction event intersects at most one cell a.s. Moreover, if $\bcal_{1}(z)$ does not intersect any cell, then it means that $z$ fell into one of the areas in white on Figure~\ref{fig:lattice_ball}.
\item If we refer to a sequence of reproduction events occurring in chronological order and such that each reproduction event intersects the previous one as a \textit{path of reproduction events}, then any path of reproduction events for which none of the corresponding balls intersect a cell is almost surely confined in one of the white areas on Figure~\ref{fig:lattice_ball}, in the sense that all the reproduction event centres fall into the same white area almost surely.
\end{enumerate}

The discretised version of the $\infty$-parent ancestral process, denoted by $(D_{t})_{t \geq 0}$, is then defined as follows. For all $t \geq 0$, let
\begin{equation*}
D_{t} := \left\{
(i,j) \in \zmath^{2} : \ccal_{i,j} \cap E_{t} \neq \emptyset
\right\}
\end{equation*}
be the set of all cells which intersect the $\infty$-parent ancestral process associated with $\mu^{\leftarrow}=\pi^{-1}\delta_1\otimes\delta_1\otimes\delta_0$ (and with initial condition $(0,0)$) at time $t$. Let $(\overleftarrow{\tau}\!_{n}^{discr})_{n \geq 0}$ be the random variables defined for all $n \geq 0$ by
\begin{equation*}
\overleftarrow{\tau}\!_{n}^{\,discr} := \min \left\{
t \geq 0 : \exists m \in \zmath, (n,m) \in D_{t}
\right\}.
\end{equation*}
Due to the structure of the grid and the size of the cells, we have the following result.

\begin{lem}\label{lem:link_taux_taudisc}
For all $n \in \nmath \setminus \{0\}$, we have
\begin{equation*}
\overleftarrow{\tau}\!_{n}^{\,discr} \leq \overleftarrow{\tau}\!_{4n}^{\,(\infty)} \quad \text{a.s}.
\end{equation*}
\end{lem}

\begin{proof}
Let $n \in \nmath \setminus \{0\}$. Let $z \in \rmath^{2}$ be the centre of the reproduction event which occurs at time $\overleftarrow{\tau}\!_{4n}^{\,(\infty)}$ and makes the $\infty$-parent ancestral process reach for the first time a point at horizontal separation $4n$ from the origin. Then, a.s. there exists $j \in \zmath$ such that
\begin{equation*}
\bcal_{1}(z) \cap \ccal_{n,j} \neq \emptyset,
\end{equation*}
and hence $\overleftarrow{\tau}\!_{n}^{\,discr} \leq \overleftarrow{\tau}\!_{4n}^{\,(\infty)}$.
\end{proof}

In all that follows, for all $(i,j), (i',j') \in \zmath^{2}$, we say that
\begin{itemize}
\item the cell $\ccal_{i,j}$ is \textit{active} at time $t$ if $(i,j) \in D_{t}$;
\item the cell $\ccal_{i,j}$ is \textit{activated} at time $t$ if $(i,j) \in D_{t}$ and $(i,j) \notin D_{t-}$;
\item the cell $\ccal_{i',j'}$ \textit{activates} $\ccal_{i,j}$ at time $t$ if there exists $s < t$ such that $\ccal_{i',j'}$ is active at time $s$, and if there exists a path of reproduction events starting from $\ccal_{i',j'}$ at time $s$, initially overlapping an area of $\ccal_{i',j'}$ containing type~$1$ individuals, and reaching $\ccal_{i,j}$ for the first time at time $t$ while not intersecting any other cell in the time interval $[s,t]$.
\end{itemize}
Notice that under this terminology, a cell can activate another one which is already active. Moreover, with probability one the only cells that the cell $\ccal_{i',j'}$ can activate are its nearest neighbours in the graph $\gcal$, that is, the cells $\ccal_{i,j}$ such that
\begin{equation*}
\sqrt{(i-i')^{2}+(j-j')^{2}} \leq \sqrt{2}.
\end{equation*}

\subsection{Comparison to the first-passage percolation problem}\label{sec:section_4_3}
Let us now compare the growth of the discretised $\infty$-parent ancestral process $(D_{t})_{t \geq 0}$ to that of the first-passage percolation problem introduced earlier. Recall that
\begin{equation*}
\overleftarrow{\tau}\!_{n}^{fpp} := \min\left\{
t \geq 0: \exists m \in \zmath, \ecal_{(0,0),(n,m)} \leq t
\right\}
\end{equation*}
is the time needed  by the process associated with the first-passage percolation problem to reach a point at horizontal separation $n$ from the origin, starting from the origin.

\begin{prop}\label{prop:comparison_discr_fpp}
The random variable $\overleftarrow{\tau}\!_{n}^{\,discr}$ is stochastically bounded from below by $\overleftarrow{\tau}\!_{n}^{\,fpp}$. That is, for all $t \geq 0$,
\begin{equation*}
\bP\left(
\overleftarrow{\tau}\!_{n}^{\,discr} \geq t
\right) \geq \bP\left(
\overleftarrow{\tau}\!_{n}^{\,fpp} \geq t
\right).
\end{equation*}
\end{prop}

\begin{proof}
We recall that for all $e \in \mathcal{G}$, $\ecal_{e} \sim \ecal\text{xp}(16\pi^{-1})$ is the passage time of edge~$e$. In order to show Proposition~\ref{prop:comparison_discr_fpp}, we show that cells are activated faster in the first-passage percolation problem than in the discretized $\infty$-parent ancestral process, and conclude by induction on the number of cells reached.

First, we observe that for both processes, cells are activated one after another a.s., and a cell cannot be activated if all its neighbours are inactive. Therefore, we can focus on the time needed for a cell to become active once (a.s. exactly) one of its neighbours becomes active. Regarding the first-passage percolation process, this time is bounded from above by $\ecal_{e}$, where $e$ is the edge connecting the active neighbouring cell to the focal cell. Regarding the discretised $\infty$-parent ancestral process, this time is bounded from below by the time needed for the cell to be intersected by a reproduction event (which is a prerequisite for the cell to become active). Such reproduction events occur at a rate bounded from above by $16\pi^{-1}$. Therefore, the time needed for the cell to become active in the discretised $\infty$-parent ancestral process is stochastically bounded from below by $\ecal_{e}$.
\end{proof}

\begin{rem}
Due to correlations between activations by neighbouring cells, it would be more difficult to construct a coupling with the first-passage percolation problem. However, the stochastic comparison we obtained is sufficient.
\end{rem}

We can now show Proposition~\ref{prop:upper_bound_growth}.

\begin{proof}(Proposition~\ref{prop:upper_bound_growth})
By Proposition \ref{prop:comparison_discr_fpp} and Lemma \ref{lem:link_taux_taudisc}, for all $n \in \nmath \setminus \{0\}$,
\begin{equation*}
\bE\big[\overleftarrow{\tau}\!_{4n}^{\,(\infty)}\big] \geq \bE\big[\overleftarrow{\tau}\!_{n}^{\,discr}\big] \geq \bE\big[\overleftarrow{\tau}\!_{n}^{\,fpp}\big].
\end{equation*}
By Lemma \ref{lem:cvg_esp_fpp}, we know that
\begin{equation*}
\lim\limits_{n \to + \infty} \frac{\bE\big[\overleftarrow{\tau}\!_{n}^{\,fpp}\big]}{4n} > 0.
\end{equation*}

Since we know that $\lim\limits_{n \to + \infty} \bE\big[\overleftarrow{\tau}\!_{n}^{\,(\infty)}\big] \times n^{-1}$ exists by Proposition~\ref{prop:lower_bound_growth}, we obtain that
\begin{equation*}
\lim\limits_{n \to + \infty} \frac{\bE\big[\overleftarrow{\tau}\!_{4n}^{\,(\infty)}\big]}{4n} > 0.
\end{equation*}
We conclude by using the fact that $x \to \overleftarrow{\tau}\!_{x}^{\,(\infty)}$ is nondecreasing by Lemma~\ref{lem:la_suite_est_croissante}.
\end{proof}

We conclude this section with the proof of Theorem~\ref{thm:speed_growth_infty_slfv}.

\begin{proof}(Theorem \ref{thm:speed_growth_infty_slfv})
By Proposition~\ref{prop:lower_bound_growth}, we know that there exists $\nu^{(\infty)} \geq 0$ such that
\begin{equation*}
\lim\limits_{x \to + \infty} \frac{\bE\big[\overleftarrow{\tau}\!_{x}^{(\infty)}\big]}{x} = \nu^{(\infty)}.
\end{equation*}
Moreover, by Proposition~\ref{prop:upper_bound_growth}, we have
\begin{equation*}
\lim\limits_{x \to + \infty} \frac{\bE\big[\overleftarrow{\tau}\!_{x}^{(\infty)}\big]}{x} > 0.
\end{equation*}
Therefore, $\nu^{(\infty)} > 0$ and we conclude by using Proposition~\ref{prop:taux_some_distr}.
\end{proof}

\section{Speed of growth of the $k$-parent SLFV}\label{sec:k-parent}
In this section, we prove Theorem~\ref{thm:speed_growth_k_slfv}. To this end, we combine Proposition~\ref{prop:backwards_k_slfv} with estimates on the probability that the reflected ancestral process $\overleftarrow{\Xi}^{(k)}$ reaches horizontal distance $x$ before time~$t$, that are derived using a weaker version of the subadditivity argument developped in Section~\ref{sec:lower_bound}.

Fix $k\geq 2$. Recall that $\overleftarrow{\Pi}^{+}$ is a Poisson point process on $\mathbb{R}_+\times \rde\times S_\mu\times (\rde)^\infty$ with intensity measure $dt \otimes dz \otimes \mu^{\leftarrow}(da,db,d\gamma)\tilde{u}_{(a,b,\gamma)}(d\mathbf{p})$ defined on the probability space $(\mathbf{\Omega}, \mathbf{\fcal}, \mathbf{P})$, and that $(\overleftarrow{\Xi}_{t}^{(k)})_{t \geq 0}$ is the $k$-parent ancestral process with initial condition $\delta_{(0,0)}$ constructed using $\overleftarrow{\Pi}^{+}$. For all $x > 0$, we set
\begin{align*}
\overleftarrow{\tau}\!_{x}^{\,(k)} &:= \min\left\{
t \geq 0 : A\left(
\overleftarrow{\Xi}_{t}^{(k)}
\right) \cap \text{H}^{x} \neq \emptyset
\right\} \\
\text{and } \overleftarrow{T}\!_{x}^{\,(k)} &:= \min\left\{
t \geq 0 : \forall t' \geq t, A\left(
\overleftarrow{\Xi}_{t'}^{(k)}
\right) \cap \text{H}^{x} \neq \emptyset
\right\}.
\end{align*}
The strategy is as follows. First, we use an approach analogous to that in Section~\ref{sec:lower_bound} to obtain that:
\begin{lem}\label{lem:cvg_tau_k}
There exists $\nu^{(k)} \geq \nu^{(\infty)}$ such that
\begin{equation*}
\lim\limits_{x \to + \infty} \frac{\bE\Big[
\overleftarrow{\tau}\!_{x}^{\,(k)}
\Big]}{x} = \nu^{(k)}.
\end{equation*}
\end{lem}
The proof of Lemma~\ref{lem:cvg_tau_k} uses an \textit{express chain} which encodes a special line of ascent (which jumps onto the location of the ``right-most parent'' each time the current position of the chain is overlapped by an event). We then use this express chain to show that:
\begin{lem}\label{lem:k_tau_T}
There exists $C_1>0$ such that for all $x\geq 0$,
\begin{equation*}
\bE\left[
\overleftarrow{\tau}\!_{x}^{\,(k)}
\right] \leq \bE\left[
\overleftarrow{T}\!_{x}^{\,(k)}
\right] \leq \bE\left[
\overleftarrow{\tau}\!_{x}^{\,(k)}
\right] + C_1.
\end{equation*}
\end{lem}
Using Lemma~\ref{lem:k_tau_T}, we shall be able to conclude that $\nu^{(k)}$ is also the limit of $x^{-1}\bE[\overleftarrow{T}\!_{x}^{\,(k)}]$, and then a simple use of the Markov inequality will give us the desired result. The proof of Theorem~\ref{thm:speed_growth_k_slfv} can be found at the end of this section.

\begin{proof}(Lemma~\ref{lem:cvg_tau_k}.)
We follow the same strategy as in the proof of Proposition~\ref{prop:lower_bound_growth} and show that the sequence $(\bE[\overleftarrow{\tau}\!_{4n\rcal_{\mu}}^{(k)}])_{n \geq 0}$ is subadditive. To do so, we introduce the following family of random variables.
For all $n \in \mathbb{N}$, let
\begin{align*}
T_{0,n}^{(k)} &:= \min\left\{t \geq 0 : A\left(
\overleftarrow{\Xi}_{t}^{(k)}
\right) \cap \hp^{4n\rcal_{\mu}} \neq \emptyset
\right\} = \overleftarrow{\tau}\!_{4n\rcal_{\mu}}^{\,(k)},
\end{align*}
and let $P_{n}^{(k)}$ be sampled uniformly at random among the points in
\begin{equation*}
A(\overleftarrow{\Xi}_{T_{0,n}^{(k)}}^{(k)}) \cap \hp^{4n\rcal_{\mu}}.
\end{equation*}
Note that in contrast to the analogous variable $P_n$ defined in Section~\ref{subsec:subadditivity}, almost surely $P_n^{(k)}$ is not at horizontal separation of exactly $4n\rcal_{\mu}$ of $(0,0)$ (\textit{i.e.}, it has a first coordinate $(P_n^{(k)})_1>4n\rcal_{\mu}$ a.s.).

Next, for all $m \in \mathbb{N}$, let $(\overleftarrow{\Xi}_{t}^{(k),m})_{t \geq 0}$ be the $k$-parent ancestral process started from $\delta_{P_{m}^{(k)}}$, constructed using only the reproduction events in $\overleftarrow{\Pi}^{+}$ occurring after time $T_{0,m}^{(k)}$. We then set for all $n\geq m$
\begin{equation*}
T_{m,n}^{(k)} := \min\left\{
t \geq 0 : A\left(
\overleftarrow{\Xi}_{t+T_{0,m}^{(k)}}^{(k),m}
\right) \cap \hp^{4n\rcal_{\mu}} \neq \emptyset
\right\}.
\end{equation*}
By construction, we have that for all $0 \leq m < n \in \nmath$,
\begin{align*}
T_{0,n}^{(k)} &\leq T_{0,m}^{(k)} + T_{m,n}^{(k)} \\
\text{and} \quad \bE\left[
T_{m,n}^{(k)}
\right] &\leq \bE\left[
\overleftarrow{\tau}\!_{4(n-m)\rcal_{\mu}}^{\,(k)}
\right].
\end{align*}
Therefore,
\begin{equation*}
\bE\left[
\overleftarrow{\tau}\!_{4n\rcal_{\mu}}^{\,(k)}
\right] \leq \bE\left[
\overleftarrow{\tau}\!_{4m\rcal_{\mu}}^{\,(k)}
\right] + \bE\left[
\overleftarrow{\tau}\!_{4(n-m)\rcal_{\mu}}^{\,(k)}
\right],
\end{equation*}
and the family $(\bE[\overleftarrow{\tau}\!_{4n\rcal_{\mu}}^{\,(k)}])_{n \geq 0}$ is subadditive.

In order to apply Fekete's subadditive lemma~\cite{fekete1923} and conclude, we now need to show that there exists $c_k > 0$ such that
\begin{equation*}
\forall n \geq 0,\quad  \bE\left[
\overleftarrow{\tau}\!_{4n\rcal_{\mu}}^{\,(k)}
\right] \leq c_k n.
\end{equation*}
To do so, we proceed again as in the proof of Proposition~\ref{prop:lower_bound_growth}, and construct an express chain for the $k$-parent SLFV, denoted $(C_{t}^{(k)})_{t \geq 0}$ and defined as follows.
\begin{itemize}
\item First we set $C_{0}^{(k)} = (0,0)$.
\item Then, for all $(t,z_{c},a,b,\gamma,(p_{n})_{n \geq 1}) \in \overleftarrow{\Pi}^{+}$, if $C_{t-}^{(k)} \in \bfrak_{a,b,\gamma}(z_{c})$, we write $z_{c} = (x_{c},y_{c})$ and $p_{n} = (x_{n},y_{n})$. We define $L$ to be the index of the (almost surely unique) point in $p_1,\ldots,p_k$ with maximal abscissa (that is, the right-most ``potential parent'' sampled). Recalling that all $p_n$ are sampled from $\bfrak_{a,b,\gamma}((0,0))$, we set
$C_{t}^{(k)} = z_{c}+p_{L}$.
\item We do nothing otherwise.
\end{itemize}
By construction (since $k\geq 2$), $\bE[x_{L}-x_{c}] > 0$, and for all $n \in \nmath \backslash \{0\}$ and $t \geq 0$, we have
\begin{equation*}
\left\{
C_{t}^{(k)} \in \text{H}^{4n\rcal_{\mu}}
\right\} \subseteq \left\{
\overleftarrow{\tau}\!_{4n\rcal_{\mu}}^{\,(k)} \leq t
\right\}.
\end{equation*}
For all $n \in \nmath$, we set
\begin{equation*}
T_{0,n}^{\text{express},(k)} := \min\left\{
t \geq 0 : C_{t}^{(k)} \in \text{H}^{4n\rcal_{\mu}}
\right\}.
\end{equation*}
Using the same arguments as in Section~\ref{sec:section_2_3}, we obtain the existence of $c_k>0$ such that for all $n \in \nmath$,
\begin{equation*}
\bE\left[
\overleftarrow{\tau}\!_{4n\rcal_{\mu}}^{\,(k)}
\right] \leq \bE\left[
T_{0,n}^{\text{express},(k)}
\right] \leq c_k n.
\end{equation*}
We can now apply Fekete's subadditive lemma and obtain the existence of $\nu^{(k)} \geq 0$ such that
\begin{equation*}
\lim\limits_{n \to + \infty} \frac{\bE\left[
\overleftarrow{\tau}\!_{4n\rcal_{\mu}}^{(k)}
\right]}{4n\rcal_{\mu}} = \nu^{(k)},
\end{equation*}
from which we deduce that
\begin{equation*}
\lim\limits_{x \to + \infty} \frac{\bE\big[
\overleftarrow{\tau}\!_{x}^{\,(k)}
\big]}{x} = \nu^{(k)}
\end{equation*}
by monotonicity. The inequality $\nu^{(k)}\geq \nu^{(\infty)}$ comes from the obvious fact that for all $x>0$,
\begin{equation*}
\bE\left[
\overleftarrow{\tau}\!_{x}^{\,(k)}
\right] \geq \bE\left[\overleftarrow{\tau}\!_{x}^{\,(\infty)}\right].
\end{equation*}
\end{proof}
\begin{proof}(Lemma~\ref{lem:k_tau_T}.) Let $x\geq 0$. By construction, $\overleftarrow{\tau}\!_{x}^{\,(k)} \leq \overleftarrow{T}\!_{x}^{\,(k)}$, which gives us the first inequality. For the second inequality, observe first that for all $z\in \mathbb{R}^2$,
\begin{equation}\label{eqn:eventually occupied}
\mathbf{P}\Big(\exists t\geq 0:\ \forall t'\geq t,\, A\left(\overleftarrow{\Xi}_{t'}^{(k)}\right)\cap H^x \neq \emptyset \, \Big|\, \overleftarrow{\Xi}_{0}^{(k)} = \delta_{z}\Big) =1,
\end{equation}
and there exists $p_0\in (0,1)$ independent of $x$ such that for all $z\in H^x$,
\begin{equation}\label{eqn:always occupied}
\mathbf{P}\Big(\forall t\geq 0,\, A\left(\overleftarrow{\Xi}_{t}^{(k)}\right)\cap H^x \neq \emptyset \, \Big|\, \overleftarrow{\Xi}_{0}^{(k)} = \delta_{z}\Big) \geq p_0.
\end{equation}
Indeed, starting the \textit{express chain} $(C_t^{(k)})_{t\geq 0}$ from $z\in \mathbb{R}^2$, we obtain a random walk whose jumps are bounded and have a positive average first coordinate. Therefore, with probability $1$ the random walk made of the first coordinate of $C^{(k)}$ drifts towards $+\infty$ and there exists $t\geq 0$ such that for all $t'\geq t$, $C^{(k)}_{t'}\in H^x$. In addition, when $z\in H^x$, the probability that this time $t$ is equal to $0$ (and thus $C^{(k)}$ has its whole trajectory in $H^x$) is positive and by the invariance properties of its dynamics, this probability is independent of the second coordinate of $z$ and is minimal for $z_1=x$. We write $p_0>0$ for this minimal probability. Since $C_t^{(k)}$ always belongs to the set of atoms of $\overleftarrow{\Xi}^{(k)}_t$ (by construction), the first part of our argument yields~\eqref{eqn:eventually occupied} and the second part yields~\eqref{eqn:always occupied}. In particular, this tells us that $\overleftarrow{T}\!_x^{\,(k)}<\infty$ almost surely. Pushing the analysis of the \textit{express chain} one step further and decomposing the trajectory of $C^{(k)}$ into excursions away from $H^x$ (whose time lengths are integrable, independently of $x$, since the first coordinate of $C^{(k)}$ performs a one-dimensional random walk with bounded jumps of positive mean) until the first time at which the trajectory of $C^{(k)}$ remains in $H^x$ for ever after (the number of such excursions being stochastically bounded by a geometric random variable with parameter $p_0$), we obtain that there exists $C_1>0$ satisfying
\begin{equation}\label{eqn:expected return}
\sup_{z\in H^x} \mathbf{E}\Big[\inf\Big\{t\geq 0:\ \forall t'\geq t,\, C_{t'}^{(k)}\in H^x\Big\}\, \big|\, C^{(k)}_0=z\Big] \leq C_1.
\end{equation}
Coming back to the case where $\overleftarrow{\Xi}^{(k)}_0=\delta_{(0,0)}$, using the strong Markov property of $\overleftarrow{\Xi}^{(k)}$ at time $\overleftarrow{\tau}\!_x^{\,(k)}$ and the fact that for every $t$, $C_t^{(k)}$ belongs to the set of atoms of $\overleftarrow{\Xi}^{(k)}_t$, enables us to write that
\begin{equation}\label{eqn:control expt}
\mathbf{E}\Big[\overleftarrow{T}\!_x^{\,(k)} - \overleftarrow{\tau}\!_x^{\,(k)}\Big] \leq C_1,
\end{equation}
which gives us the desired result.
\end{proof}

Let us conclude this section with the proof of Theorem~\ref{thm:speed_growth_k_slfv}.
\begin{proof} (Theorem~\ref{thm:speed_growth_k_slfv}.) By Lemmas~\ref{lem:cvg_tau_k} and \ref{lem:k_tau_T}, we have
$$
\lim\limits_{x \to + \infty} \frac{\bE\Big[\overleftarrow{T}\!_{x}^{\,(k)}\Big]}{x} = \nu^{(k)}.
$$
Therefore, for every $\varepsilon>0$, there exists $x_0(\varepsilon)>0$ such that for all $x\geq x_0(\varepsilon)$,
\begin{equation}\label{eqn:inequ}
\bE\Big[\overleftarrow{T}\!_{x}^{\,(k)}\Big] \leq \big(\nu^{(k)} +\varepsilon\big)x.
\end{equation}
Fix $\varepsilon>0$ and let $x\geq x_0(\varepsilon)$. By Proposition~\ref{prop:backwards_k_slfv}, we have for every $t>0$
\begin{align*}
\proba\left(
\limeps V_{\epsilon}^{-1} \int_{\bcal_{\epsilon}((x,0))} w_{t}^{(k)}(z)dz = 0
\right) &= \mathbf{P}\left(
A\left(
\overleftarrow{\Xi}_{t}^{(k)}
\right) \cap \text{H}^{x} = \emptyset
\right) \\
&\leq \mathbf{P}\left(
\exists t' \geq t, A\left(
\overleftarrow{\Xi}_{t'}^{(k)}
\right) \cap \text{H}^{x} = \emptyset
\right) \\
&= \mathbf{P}\left(
\overleftarrow{T}\!_{x}^{(k)} > t
\right).
\end{align*}
We conclude by using the Markov inequality together with \eqref{eqn:inequ}.
\end{proof}

\section{Numerical simulations}\label{sec:numerical_simulations}
The goal of this section is to obtain an approximation for the limiting speed of growth $(\nu^{(\infty)})^{-1}$ by means of numerical simulations, and compare it to the lower bound given by the proof of Theorem~\ref{thm:speed_growth_infty_slfv}. To do so, we use the fact that
\begin{equation*}
(\nu^{(\infty)})^{-1} = \left(\lim\limits_{x \to + \infty} x^{-1}\bE[\overleftarrow{\tau}\!_{x}^{\,(\infty)}]\right)^{-1},
\end{equation*}
which is a consequence of Theorem~\ref{thm:speed_growth_infty_slfv} combined with Proposition~\ref{prop:taux_some_distr}. Indeed, the $\infty$-parent ancestral process is easier to simulate than its forwards-in-time counterpart, since it jumps at a finite rate at any time, and since there are no border effects to take into account while simulating the process on an appropriately chosen compact subset of $\mathbb{R}^{2}$.

We focus on the case in which all ellipses have the same shape parameters. In order to be able to compare the speed of growth of the occupied regions in $\infty$-parent SLFV with different shape parameters, we assume that any given location $z \in \mathbb{R}^{2}$ is affected by a reproduction event at rate~$1$. Therefore, we take
\begin{equation*}
\mu(da,db,d\gamma) = V_{a_0,b_0}^{-1} \delta_{a_{0}}(da) \otimes \delta_{b_{0}}(db) \otimes \delta_{0}(d\gamma),
\end{equation*}
where $a_{0} \in (0,+\infty)$ and $b_{0}$ is chosen such that the volume of the corresponding ellipse is equal to $\pi$.

For $9$ different values of $a$ ranging from $0.33$ to $3$, we simulate $30$ $\infty$-parent ancestral processes with initial condition $\{(0,0)\}$ and parameters $(a,b,0)$, where $b$ is chosen as stated above. This ensures that we compare $\infty$-parent SLFV processes for which reproduction events have the same scale.
\begin{figure}[t]
\centering
\includegraphics[width = 0.7\linewidth]{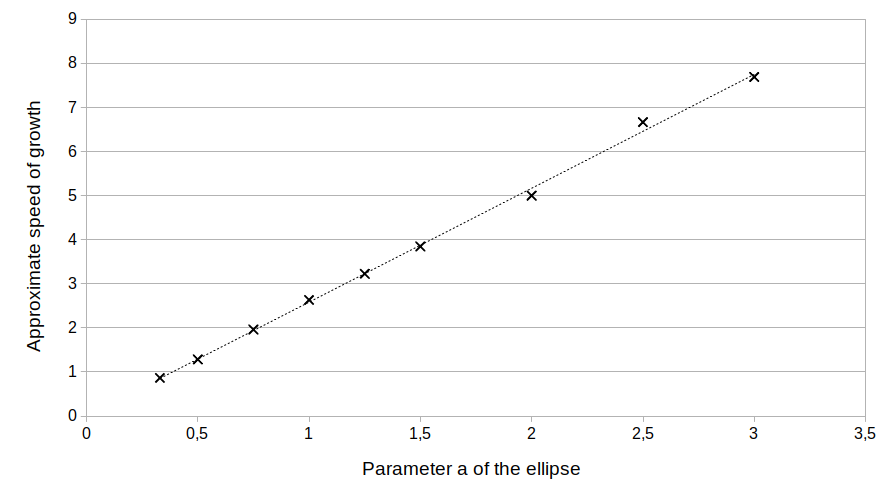}
\caption{Approximate speed of growth of the occupied area in the $\infty$-parent SLFV process, as a function of $a$. For each value of $a$, $30$ $\infty$-parent ancestral processes with parameters $(a,b,0)$ were simulated, in order to compute the expected time to reach horizontal distance $x$ for large values of $x$. The crosses indicate the approximate values, and the dotted line corresponds to the linear function $\nu^{-1}(a) = 2.58a$.}\label{fig:speed of growth}
\end{figure}
The results can be found in Figure~\ref{fig:speed of growth}. The numerical simulations show that the speed of growth is a linear function of $a$ (over the range of $a$-values considered). This speed is around $2.6$ times higher than the lower bound obtained in Section~\ref{sec:lower_bound} (equal to $a$ when $\gamma = 0$, see Remark~\ref{rem:speed in example}), which we had initially conjectured to be the limiting speed of the process.

\begin{figure}[t]
\centering
\begin{subfigure}[b]{0.2\textwidth}
\includegraphics[width = 0.8\linewidth]{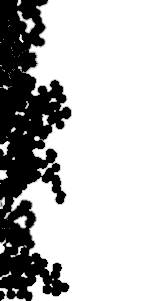}
\end{subfigure}
\begin{subfigure}[b]{0.2\textwidth}
\includegraphics[width = 0.8\linewidth]{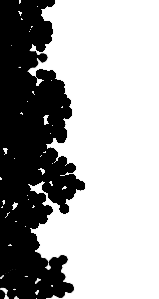}
\end{subfigure}
\begin{subfigure}[b]{0.2\textwidth}
\includegraphics[width = 0.8\linewidth]{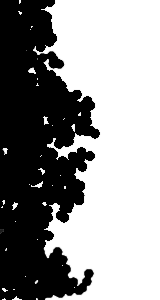}
\end{subfigure}
\begin{subfigure}[b]{0.2\textwidth}
\includegraphics[width = 0.8\linewidth]{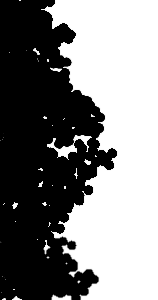}
\end{subfigure}\vspace{0.2cm}
\caption{Illustration of the growth dynamics of the occupied region in the $\infty$-parent SLFV, with $(a,b,\gamma) = (1,1,0)$. The images represent the same $\infty$-parent SLFV at four instants $0 < t_{1} < t_{2} < t_{3} < t_{4}$. The black area represents the area occupied by real individuals, and the white area is empty. The expansion starts from the left-most part of the image, and goes towards the right-most part of the image.}\label{fig:spikes}
\end{figure}

Numerical simulations of the $\infty$-parent SLFV itself suggest that the growth of the process is driven by ``spikes'' in the expansion direction, see Figure~\ref{fig:spikes} for an example. These spikes are relatively rare, but then they thicken and grow sideways, bridging the gap with the rest of the population. Therefore, they can increase the speed of growth of the process significantly. Moreover, a ``spike-driven'' growth would have important implications in terms of genetic diversity at the front edge. The study of the contribution of the spikes to the overall growth will be the object of a follow-up study.

\smallskip
\begin{acknowledgements}
\quad This work was partly supported by a grant from the British Royal Society and by the \textit{chaire} programme ``Mathematical Modelling and Biodiversity'' of Veolia Environnement -- Ecole Polytechnique -- Museum National d'Histoire Naturelle -- Fondation X. AL would like to thank Matt Roberts, Laetitia Colombani and Carl Graham for insightful discussions. AL and AV thank the two Reviewers and the Associate Editor for their very careful reading and their relevant suggestions.
\end{acknowledgements}

\bibliographystyle{plain}
\bibliography{bibliographie_vitesse_croissance_infty_slfv}

\appendix
\section{Geometrical properties of ellipses}\label{appendix:geom_ellipse}
In this section, we show some geometrical properties of ellipses, which are used in other sections. In all that follows, let $z_{c} = (x_{c},y_{c}) \in \rde$, $(a,b) \in (0,+\infty)$ and $\gamma \in (-\pi/2,\pi/2)$. Recall that $\bfrak_{a,b,\gamma}(z)$ is the ellipse defined by:
\begin{equation*}
\bfrak_{a,b,\gamma}(z_{c}) = \left\{\begin{pmatrix}
x_{c} \\
y_{c}
\end{pmatrix}
+ A_{\gamma} \begin{pmatrix}
ar\cos(\theta) \\
br \sin(\theta)
\end{pmatrix} : r \in [0,1], \theta \in [0,2\pi)
\right\}
\end{equation*}
where
\begin{equation*}
A_{\gamma} = \begin{pmatrix}
\cos(\gamma) & - \sin(\gamma) \\
\sin(\gamma) & \cos(\gamma)
\end{pmatrix}.
\end{equation*}

The first lemma gives the maximal \textit{horizontal separation} between a point in the ellipse and its centre. This result is used in Section \ref{sec:lower_bound} to construct the express chain.
\begin{lem}\label{lem:ellipse_1}
Let $f : [0,1] \times [-\pi,\pi) \to \rmath$ be the function defined by
\begin{equation*}
\forall\, (r,\theta) \in [0,1] \times [-\pi,\pi),
f(r,\theta) = ar\cos(\theta)\cos(\gamma)-br\sin(\theta)\sin(\gamma).
\end{equation*}
Then, $f$ reaches its maximum for
\begin{equation*}
(r_{max},\theta_{\max}) = \left(1, \arctan\left(
-\frac{b}{a}\tan(\gamma)
\right)\right),
\end{equation*}
and
\begin{equation*}
f(r_{max},\theta_{max}) = \dhor.
\end{equation*}
\end{lem}

\begin{proof}
First, it is obvious that $r_{max} = 1$. Moreover, $\cos(\theta_{max})$ must be of the same parity as $\cos(\gamma)$, and $\sin(\theta)$ must be of opposite parity from $\sin(\gamma)$. As $\gamma \in (-\pi/2,\pi/2)$, we obtain that $\theta_{\max} \in (-\pi/2,\pi/2)$ too (that is, it cannot take the values $\pm \pi/2$).

The function $f_{\theta} : (-\pi/2,\pi/2) \to \rmath$ such that for all $\theta \in (-\pi/2,\pi/2)$,
\begin{equation*}
f_{\theta}(\theta) := a\cos(\theta)\cos(\gamma)-b\sin(\theta)\sin(\gamma)
\end{equation*}
is differentiable, and for all $\theta \in (-\pi/2,\pi/2)$,
\begin{equation*}
f'_{\theta}(\theta) = -a\sin(\theta)\cos(\gamma)-b\cos(\theta)\sin(\gamma).
\end{equation*}
Therefore,
\begin{align*}
 f'_{\theta}(\theta_{max}) = 0 &\Longleftrightarrow  \theta_{max} = \arctan\left(-\frac{b}{a}\tan(\gamma)\right).
\end{align*}
Moreover, since $\cos(\gamma) > 0$, we can write
\begin{align*}
&a\cos(\gamma)\cos(\theta_{max}) - b\sin(\gamma)\sin(\theta_{max}) \\
&= a\cos(\gamma)\cos\left(\arctan\left(-\frac{b}{a}\tan(\gamma)\right)\right) - b \sin(\gamma)\sin\left(\arctan\left(-\frac{b}{a}\tan(\gamma)\right)\right) \\
&= a \cos(\gamma) \frac{1}{\sqrt{1 + \frac{b^{2}}{a^{2}}\tan^{2}(\gamma)}}
+ b \sin(\gamma) \frac{\frac{b}{a}\tan(\gamma)}{\sqrt{1 + \frac{b^{2}}{a^{2}}\tan^{2}(\gamma)}} \\
&= \frac{a^{2}\cos^{2}(\gamma)}{a\cos(\gamma) \sqrt{1 + \frac{b^{2}}{a^{2}}\tan^{2}(\gamma)}}
+ \frac{1}{a\cos(\gamma)}\frac{b^{2}\sin^{2}(\gamma)}{\sqrt{1 + \frac{b^{2}}{a^{2}}\tan^{2}(\gamma)}} \\
&= \frac{a^{2}\cos^{2}(\gamma)+b^{2}\sin^{2}(\gamma)}{\sqrt{a^{2}\cos^{2}(\gamma)+b^{2}\sin^{2}(\gamma)}} \\
&= \sqrt{a^{2}\cos^{2}(\gamma)+b^{2}\sin^{2}(\gamma)},
\end{align*}
which concludes the proof.
\end{proof}

The second lemma means that the mean horizontal separation of a point in the ellipse from its centre is equal to $0$. Therefore, when a point is affected by a reproduction event, the mean horizontal separation of the centre of the corresponding ellipse from the point is equal to $0$. That is:
\begin{lem}\label{lem:ellipse_2}
Let $Z = (x+X,y+Y)$ be sampled uniformly at random in the ellipse $\bfrak_{a,b,\gamma}(z)$. Then,
\begin{equation*}
\bE[X] = 0.
\end{equation*}
\end{lem}

\end{document}